\documentclass{JTPaper}

\usepackage{xr-hyper}
\usepackage{hyperref}

\title{A general connected sum formula for the families Seiberg-Witten invariant}
\author{Joshua Tomlin}
\address{School of Mathematical Sciences, University of Adelaide, Adelaide SA 5005, Australia}
\email{joshua.tomlin@adelaide.edu.au}

\begin{document}

\pagenumbering{arabic}
\maketitle
\markright{{\tiny A GENERAL CONNECTED SUM FORMULA FOR THE FAMILIES SEIBERG-WITTEN INVARIANT}}

\begin{abstract}
    In ordinary Seiberg-Witten theory, there are well known connected sum formulae such as the vanishing formula and the blow up formula. For families Seiberg-Witten theory, there are results such as Liu's families blow-up formula and Baraglia-Konno's gluing formula, but these have limited uses. In this paper, we prove a general connected sum formula which incorporates these results. This is subsequent work to a previous paper \cite{Part1} in which we proved a connected sum formula for the Bauer-Furuta invariant.
\end{abstract}

\section{Introduction}

Since 1994, Seiberg-Witten theory has been a central tool in the study of smooth 4-manifolds. In recent years, attention has been focused on studying the Seiberg-Witten theory of 4-manifold families. Some striking results appear even in the simplest case of one parameter families such as mapping tori \cite{SmoothIsotopyObstruction}. In ordinary Seiberg-Witten theory, the Seiberg-Witten invariant of a connected sum is well understood through the vanishing formula \cite{SalamonVanishingTheorem} and the blow-up formula \cite{GeneralisedBlowupFormula}. There are some similar results for families such as Liu's families blow-up formula \cite{FamiliesBlowupFormula} and Baraglia-Konno's gluing formula \cite{BaragliaSWGluingFormula}, however these results only apply to limited situations. 

Let $\pi : E \to B$ be a locally trivial family of smooth, oriented, closed 4-manifolds $X$ over a compact, smooth, connected base manifold $B$. Let $n$ be the number of connected components of $X$ and write $X_b = \pi\inv(b)$ to denote the fibre of $E$ over $b \in B$. Fix a smooth family of metrics $g$ and a smooth family of \spinc structures $\sfrak$ on the vertical tangent bundle $E/B$ of $E$. Write $W \to B$ to denote the corresponding family of spinor bundles with $W = W^+ \oplus W^-$ the associated splitting into positive and negative spinors. Let $\Jcal \to B$ be the Jacobian torus bundle with fibre $H^1(X_b ; \R)/H^1(X_b ; 2\pi \Z)$ for $b \in B$. The Seiberg-Witten monopole map $\mu$ is a map between Hilbert bundles over $\Jcal$ which takes the form 
\begin{align*}
    \mu : L^2_k(E, W^+ \oplus T^* E/B \oplus \R^n) &\to L^2_{k-1}(E, W^- \oplus \Lambda^2_+(T^* E/B) \oplus \R) \oplus \Hcal^1(\R).
\end{align*}
The map is described by the Seiberg-Witten equations so that a level set $\mu\inv(0)$ is the set of Seiberg-Witten monopoles. Here $k \geq 4$ is some fixed parameter. Since $E$ is compact, integration determines an $L^2$-metric on sections of spinors and forms, and $L^2_k$ denotes the Sobolev space of sections with $k$ weak derivatives in $L^2$. The bundle $\Hcal^1(\R) \to \Jcal$ is trivial with standard fibre $H^1(X ; \R)$, viewed as the space of harmonic one-forms on $X$. On the above bundles, there is a fibre-preserving $\bT^n = (S^1)^{\times n}$ torus action of constant gauge transformations under which $\mu$ is equivariant. For an explicit description see Example \ref{Ex:FamMu} and \cite{Part1}.

This paper is subsequent work to \cite{Part1}, in which we proved a connected sum formula for the families Bauer-Furuta invariant. The Bauer-Furuta invariant is a stable cohomotopy refinement of the Seiberg-Witten invariant.  Through approximation by finite dimensional subspaces, the Seiberg-Witten monopole map $\mu$ defines a Bauer-Furuta class 
\[[\mu] \in \pi^{b^+}_{S^1, \Ucal}(\Jcal, \ind_J D).\]
Here $\Ucal$ is the $S^1$-universe $\Ucal = L^2_{k-1}(X, W^- \oplus \Lambda^2_+(T^* X) \oplus \R^n) \oplus H^1(X ; \R)$, $\ind_\Jcal D$ is the virtual index bundle of the Dirac operator and $b^+$ is the dimension of $H^2_+(X ; \R)$. Below is the primary result that we will be using.
\begin{theorem}[\cite{Part1} Theorem 6.8]\label{T:FamBFCSF}
    For $j \in \{1,2\}$, let $E_j \to B$ be a 4-manifold family equipped with a \spinc structure $\sfrak_j$ on the vertical tangent bundle. Let $i_j : B \to E_j$ be a section with normal bundle $V_j$ and assume that $\vphi : V_1 \to V_2$ is an orientation reversing isomorphism satisfying 
    \begin{align*}
        \vphi(i_1^*(\sfrak_{E_1})) \cong i_2^*(\sfrak_{E_2}).
    \end{align*} 
    Then the families Bauer-Furuta class of the fiberwise connected sum $E = E_1 \#_B E_2$ is
    \begin{align}
        [\mu_{E}] &= [\mu_{E_1}] \wedge_\Jcal [\mu_{E_2}].
    \end{align}
\end{theorem}
In this paper, we will derive a general connected sum formula for the families Seiberg-Witten invariant. This is accomplished by recovering the families Seiberg-Witten invariant from the families Bauer-Furuta invariant and applying Theorem \ref{T:FamBFCSF}. The families Seiberg-Witten invariant is usually defined in terms of the moduli space of $\mu$ for a suitably chosen perturbation. However, Baraglia-Konno have previously described a reformulation of the families Seiberg-Witten invariant in equivariant cohomology \cite{BaragliaKonnoBFandSW}. After reviewing their work, we derive a localisation formula for the Bauer-Furuta class of a smash product of two monopole maps in Section \ref{S:Localisation}. Following this, we perform the necessary cohomology calculations in  Section \ref{Ch:FSW-CSF} which lead to the final formula.
\begin{theorem}[Families Seiberg-Witten Connected Sum Formula]\label{Tintro:FSW-CSF}
    For $j \in \{1,2\}$, let $E_j \to B$ be a 4-manifold family equipped with a \spinc structure $\sfrak_j$ on the vertical tangent bundle. Let $i_j : B \to E_j$ be a section with normal bundle $V_j$ and assume that $\vphi : V_1 \to V_2$ is an orientation reversing isomorphism satisfying 
    \[\vphi(i_1^*(\sfrak_{E_1})) \cong i_2^*(\sfrak_{E_2}).\] 
    Set $E = E_1 \#_B E_2$ and let $\phi_2$ be a chamber for $\mu_2$ with $\phi = (0, \phi_2)$ defining a chamber for $\mu$. Then 
    \begin{align}
        SW^{\mu, \phi}_m &= SW^{\mu_2,\phi_2}(x^m \degSone(\mu_1)).
    \end{align}
\end{theorem}
Let $E_1, E_2 \to B$ be families as in Theorem \ref{T:FSW-CSF} with fibre $X_1$ and $X_2$ respectively. Write $H_i^+ \to B$ to denote the $b^+(X_i)$-dimensional bundle of self-dual harmonic 2-forms on $X_i$. Assume that the virtual dimension of the moduli space for $X_1$ is given by $d(X_1, \sfrak_{X_1}) = 2m \geq 0$. In \cite[Theorem 4.4]{BaragliaSWGluingFormula}, Baraglia-Konno derived a Seiberg-Witten connected sum formula under the assumptions that 
\begin{enumerate}
    \item $b_1(X_2) = 0$
    \item $c_1(\sfrak_{X_2})^2 = \sigma(X_2)$
    \item $b^+(X_1) + b^+(X_2) > \dim B + 1$
\end{enumerate}
These assumptions where necessary to perform delicate gluing arguments. The first assumption implies $\Jcal(X_2) = 0$ and simplifies the structure of the reducible monopoles on $X_2$. The second assumption ensures that the index of the Dirac operator on $X_2$ is zero, which implies that $\degSone \mu_1 = e(H_1^+)$, the Euler class of $H_1^+$. Assumption (3) guarantees that there is a unique choice of chamber for $\mu$. Given these three assumptions, Theorem \ref{Tintro:FSW-CSF} reduces to 
\begin{align*}
    SW^{\mu, \phi}_m(\alpha) &= SW(X_1, \sfrak_{X_1}) \cdot \< \alpha \cup e(H_1^+), [B]\>.
\end{align*}
Here $\alpha \in H^s(B)$ where $s = \dim B - b^+(X_2)$. Hence Theorem (\ref{Tintro:FSW-CSF}) is a significant generalisation of Baraglia-Konno's formula. It is able to accommodate chambers and handle cases where $\Jcal(X_2)$ and $\degSone(\mu_1)$ are non-trivial. Moreover, in the case that one of the fibres $X_i$ is $\CPtwobar$, the families blow-up formula \cite{FamiliesBlowupFormula} is recovered.

\section{Families Seiberg-Witten invariants}\label{Ch:CohomologicalFormulation}

\subsection{Generalised monopole maps}

Let $H_\C, H_\C' \to B$ be locally trivial, separable, complex Hilbert bundles over a compact, smooth, oriented manifold $B$. Let $S^1$ act orthogonally on $H_\C$ and $H_\C'$ by fibrewise multiplication. Similarly, let $H_\R, H_\R' \to B$ be real Hilbert bundles on which $S^1$ acts trivially and set 
\begin{align*}
    H' &= H'_\C \oplus H'_\R \nonumber\\ 
    H &= H_\C \oplus H_\R.
\end{align*}
In \cite{BaragliaKonnoBFandSW}, Baraglia-Konno define a (generalised) families monopole map to be a smooth equivariant bundle map $f : H' \to H$ such that
\begin{enumerate}
    \item (Fredholm) The map $f$ decomposes as $f = l + c$ where $l : H' \to H$ is a smooth, $S^1$-equivariant, fibrewise linear Fredholm map and $c : H' \to H$ is an $S^1$-equivariant family of smooth compact maps, not necessarily linear.
    \item (Boundedness) The pre-image under $f$ of any disk bundle in $H$ is contained in a disk bundle in $H'$.
    \item (Injectivity) Since $l$ is linear and equivariant, we can write $l = l_\C \oplus l_\R : H'_\C \oplus H'_\R \to H_\C \oplus H_\R$. We require that $l_\R : H'_\R \to H_\R$ is injective.
    \item (Vanishing) For any $v \in H'_\R$, we have $c(v) = 0$.
    \item (Quadratic) There exists a smooth bilinear map $q : H' \times_B H' \to H$ such that $c(v) = q(v,v)$.
\end{enumerate}
\begin{example}\label{Ex:FamMu}
    Let $E \to B$ be a 4-manifold family with fibre $X$. Choose a \spinc structure on the vertical tangent bundle $T(E/b)$ and fix an integer $k \geq 4$. Then as in \cite[\S 3.2]{Part1}, the families Seiberg-Witten monopole map $\mu : \Acal \to \Ccal$ is a map of bundles over the Jacobian bundle $\Jcal$ with 
    \begin{align*}
        \Acal_\C = L^2_k(E, W^+), &\qquad \Acal_\R = L^2_k(E, T^*(E/B)) \oplus \R^n \nonumber\\
        \Ccal_\C = L^2_k(E, W^-), &\qquad \Ccal_\R = L^2_k(E, \Lambda^2_+ T^*(E/B) \oplus \R) \oplus \Hcal^1(\R).
    \end{align*}
    For a reference connection $A_0$, recall that an element $\theta \in \Jcal$ defines another connection $A_\theta = A_0 + i\theta$. For each $\theta \in \Jcal$, the map $\mu^\theta$ is defined by
    \begin{align}
        \mu^\theta(\psi, a, f) &= (D_{A_\theta + ia}\psi, -i F^+_{A_0 + ia} + d^+ a + iF^+_{A_0}, d^* a, \pr(a)).
    \end{align}
    Note that this formula differs slightly from the definition of $\mu$ in \cite{Part1} in that we have added a constant $+ i F^+_{A_0}$ term. This only shifts the level sets of $\mu$, but the change is made to ensure that $\mu$ satisfies the vanishing condition above. There is a decomposition $\mu^\theta = l^\theta + c^\theta$ with 
    \begin{align*}
        l^\theta(\psi, a, f) &= (D_{A_\theta}\psi, d^+ a, d^* a + f, \pr(a))\\
        c^\theta(\psi, a, f) &= (ia \cdot \psi, i\sigma(\psi), 0,0).
    \end{align*}
    In particular $l^\theta_\C(\psi) = D_{A_\theta}\psi$ and $l^\theta_\R(a, f) = (d^+ a, d^* a + f, \pr(a))$. It is clear that $l_\R$ is injective since if $d^+ a = 0$ and $d^* a = 0$, then $a$ is harmonic and $\pr(a) = a$. The fact that $c$ satisfies the quadratic condition was shown in \cite{BF1}.
\end{example}
Since $l_\R : H'_\R \to H_\R$ is injective, we can identify $l_\R(H'_\R) \subset H_\R$ with $H'_\R$ and assume that $l_\R$ is the inclusion. Let $H^+ = H'_\R/H_\R$, which is a vector bundle that can be identified with the cokernel of $l_\R$. Since $l_\R$ is injective and Fredholm, we have that the virtual index bundle $\ind l_\R$ is equal to $H^+$. Let $b^+$ denote the index of $l_\R$ so that $\dim(H^+) = b^+$.
\begin{definition}\label{D:ChamberFSW}
    A chamber for a families monopole map $f : H' \to H$ is a homotopy class of section $\eta : B \to (H_\R - H'_\R)$. Denote the set of chambers of $f$ by $\CH(f)$.
\end{definition}
\begin{remark}\label{R:transverseAssumps}
    In order for a map $f : H' \to H$ to be a generalised families monopole map, there is a further transversality assumption that needs to be made about the chambers of $f$ (see (M6) in \cite{BaragliaKonnoBFandSW}). In ordinary Seiberg-Witten theory, this condition is fulfilled by the existence of regular perturbations.
\end{remark}
Note that $\CH(f)$ can sometimes be empty. The set of chambers $\CH(f)$ is equivalent to the set of non-vanishing homotopy classes of sections of $H^+$, which is also equivalent to the homotopy classes of sections of the unit sphere bundle $S(H^+)$. If $b^+ = 0$, then $H^+ = 0$ and no chamber exists, which is analogous to the case that $H^2_+(X ; \R) = 0$ in ordinary Seiberg-Witten theory. If $b^+ \geq 1 + \dim B$, then a chamber is guaranteed to exist and if $b^+ > 1 + \dim B$, there is a unique choice of chamber.

Suppose that $\eta : B \to (H_\R - H_\R')$ is a representative of a chamber and assume that $\eta$ is transverse to $f$. Let $\tMcal_\eta = f\inv(\eta)$, which is a smooth manifold equipped with a free $S^1$-action. It can be shown using the monopole map axioms above that $\tMcal_\eta$ is compact and finite dimensional with dimension $\ind(l) + \dim B$. The quotient $\Mcal_\eta = \tMcal_\eta/S^1$ is thus a compact $\ind(l) + \dim B - 1$ dimension manifold. It is also a fibre bundle over $B$, with projection map $\pi_{\Mcal_\eta} : \Mcal_\eta \to B$ defined by $\pi_{\Mcal_\eta} = p \circ f$. We will assume that the determinant line bundle $\det(\ind l_R)$ is orientable, although this assumption is not strictly necessary. Fix an orientation of $\det(\ind l_R)$, which uniquely determines an orientation of $\Mcal_\eta$ \cite[Lemma 2.6]{BaragliaKonnoBFandSW}. Let $\Lcal \to \Mcal_\eta$ be the associated complex line bundle of the principle $S^1$-bundle $\tMcal_\eta \to \Mcal_\eta$ and denote by $c_1(\Lcal)$ its first chern class. 

\begin{definition}
    For any integer $m \geq 0$ and $[\eta] \in \CH(f)$, the $m$th families Seiberg-Witten invariant of a monopole map $f : H' \to H$ is 
    \begin{align}\label{E:SWmDef}
        \SW^{f,\eta}_m &= (\pi_{\Mcal_\eta})_*(c_1(\Lcal)^m) \in H^{2m-(\ind l - 1)}(B ; \Z).
    \end{align}
\end{definition}
Here the `wrong-way' map $(\pi_{\Mcal_\eta})_* : H^*(\Mcal_\eta ; \Z) \to H^{* - (\ind l - 1)}(B ; \Z)$ is defined using \Poincare duality and the pushforward map in homology. In de Rahm cohomology, this map is given by integration along the fibres. Lemma 2.9 in \cite{BaragliaKonnoBFandSW} illustrates that $\SW^{f,\eta}_m$ depends only on the homotopy class of $\eta$. In the case that $\det(\ind l_\R)$ is not orientable, a local coefficient system can be used instead of integer coefficients.

\subsection{Finite dimensional monopole maps}

Let $V, V' \to B$ be finite dimensional complex vector bundles of rank $a$ and $a'$ respectively on which $S^1$ acts by scalar multiplication. Similarly, let $U, U' \to B$ be real vector bundles of rank $b$ and $b'$ on which $S^1$ acts trivially. Let $S_{V,U} = S(\R \oplus V \oplus U)$ denote the unit sphere in $\R \oplus V \oplus U$ and note that $S_{V,U} = S_V \wedge S_U$. Identify $B \subset S_{V, U}$ with the image of the section at infinity. Let $f : S_{V', U'} \to S_{V, U}$ be an equivariant map which fixes $B$. 

For any finite dimensional complex vector bundle $A \to B$, define the stabilisation of $V$ by $A$ to be the Whitney sum $V \oplus A$. Similarly, for $C \to B$ a finite dimensional real vector bundle, the stabilisation of $U$ by $C$ is $U \oplus C$. An immediate consequence is that $S_{V' \oplus A, U' \oplus C} = S_{V', U'} \wedge S_{A, C}$, hence the stabilisation of $f$ by $A \oplus C$ is 
\[
    f \wedge \id_{A \oplus C} : S_{V', U'} \wedge S_{A, C} \to S_{V, U} \wedge S_{A, C}. 
\]
One way to obtain a map $f : S_{V', U'} \to S_{V, U}$ is to restrict a monopole map $g : H' \to H$ to a finite dimensional subspace $V' \oplus U' \subset H'$. In this case, the stable homotopy class of $f = g|_{S_{V', U'}}$ is independent of the choice of subspace $V' \oplus U'$ \cite[Definition 2.12]{Part1}. That is, for any other subspace $(V' \oplus A) \oplus (U' \oplus C)$, it is the case that $g|_{S_{V' \oplus A, U' \oplus C}}$ is stably homotopic to $g|_{S_{V', U'}} \wedge \id_{A \oplus C}$.

For $f : S_{V',U'} \to S_{V, U}$, define virtual bundles $D = V' - V \in K^0(B)$ and $H^+ = U - U' \in KO^0(B)$ with $d = a' - a$ and $b^+ = b - b'$. In the case that $f$ arises from the finite dimensional approximation of a Seiberg-Witten monopole map, $D$ is the virtual index bundle of the Dirac operator and $H^+$ is trivial with fibre $H^2_+(X ; \R)$. In this setting, $f|_{U'}$ is injective and $U'$ is identified as a subbundle of $U$ with $f|_{U'}$ the inclusion.
\begin{definition}\label{D:FDMonopoleMap}
    An $S^1$-equivariant map $f : (S_{V', U'}, B_{V', U'}) \to (S_{V, U}, B_{V, U})$ is a finite dimensional monopole map if $f|_{U'}$ is injective.
\end{definition}
Let $f : S_{V', U'} \to S_{V, U}$ be a finite dimensional monopole map. After stabilisation, we can assume that $U = U' \oplus H^+$ and that $H^+$ is a genuine vector bundle, with $f|_{U'}$ the inclusion.
\begin{definition}
    A chamber for a finite dimensional monopole map $f$ is a homotopy class of section $\phi : B \to (U - U')$.
\end{definition}
By the discussion above, a chamber $\phi : B \to (U - U')$ is equivalent to the homotopy class of a non-vanishing section of $H^+$, or a homotopy class of section of the unit sphere bundle $S(H^+) \to B$. Thus this notion of chamber is equivalent to Definition \ref{D:ChamberFSW} and we can denote set of chambers for $f$ by $\CH(f)$. 

Fix a chamber $\phi \in \CH(f)$. In the same manner as before, we will assume that $f$ is transverse to $\phi$ so that $\tMcal_\phi = f\inv(\phi(B))$ is a compact smooth submanifold of $V' \oplus U'$. Since $f$ is equivariant and $\tMcal_\phi$ is disjoint from $U'$, there is an induced free $S^1$-action on $\tMcal_\phi$. This means that $\Mcal_\phi = \tMcal_\phi/S^1$ is a compact smooth manifold of dimension $(2d - b^+) + \dim(B) - 1$. The space $\Mcal_\phi$ is fibred over $B$ with $\pi_{\Mcal_\phi} : \Mcal_\phi \to B$ the projection and $\tMcal_\phi \to \Mcal_\phi$ is a principle $S^1$-bundle with associated complex line bundle $\Lcal = \tMcal_\phi \times_{S^1} \C \to \Mcal$. Assume that $\det(\ind l|_{U'})$ is oriented, which determines an orientation on $\Mcal_\phi$.
\begin{definition}\label{D:FDSWFormula}
    For any integer $m \geq 0$ and $[\phi] \in \CH(f)$, the $m$th families Seiberg-Witten invariant of a finite dimensional monopole map $f$ is 
    \[\SWfphi_m = (\pi_{\Mcal_\phi})_*(c_1(\Lcal)^m) \in H^{2m - (2d - b^+ - 1)}(B ; \Z).\]
\end{definition}
\begin{proposition}[\cite{BaragliaKonnoBFandSW} Proposition 2.18]
    The invariant $\SWfphi_m$ depends only on the homotopy class of the chamber $\phi$ and the stable homotopy class of the finite dimensional monopole map $f$.
\end{proposition}
Further, it was noted above that a finite dimensional monopole map can be obtained by restricting a monopole map to a finite dimensional subspace. To do this, let $f = l + c : H' \to H$ be a monopole map and let $V \oplus U \subset H$ be an admissible $S^1$-invariant subbundle as in \cite[Definition 2.10]{Part1}. Set $V' = l_\C\inv(V)$ and $U' = l_\R\inv(U)$. Then the restriction $f|_{V' \oplus U'}$ extended to a map of sphere bundles $f_{V', U'} : S_{V',U'} \to S_{V,U}$ is a finite dimensional monopole map. Let $\eta : B \to (H - H')$ be a chamber for $f$ which uniquely determines a chamber $\phi : B \to (U - U')$ for $f_{V', U'}$.
\begin{theorem}[\cite{BaragliaKonnoBFandSW} Theorem 2.24]
    The Seiberg-Witten invariants $SW^{f, \eta}_m$ and $SW^{f_{V', U'}, \phi}_m$ are equal.
\end{theorem}

\newpage
\section{Formulation in equivariant cohomology}

In \cite{BaragliaKonnoBFandSW}, Baraglia-Konno reformulate the Seiberg-Witten invariant $\SWfphi$ as a map in $S^1$-equivariant cohomology. One advantage of doing so is that this formulation describes the Seiberg-Witten invariant without explicitly referring to the moduli space $\Mcal$. Instead, it focuses attention on the homotopy class of $f$. Consequently, this definition complements the families Bauer-Furuta invariant as defined in \cite[Definition 2.12]{Part1}.

\subsection{Equivariant cohomology}

Let $G$ be a continuous Lie group and $X$ a topological space with a continuous left $G$-action. For ease of notation, we will assume $\Z$ coefficients for all cohomology groups following, although orientability issues can be overcome by using local coefficient systems. Recall that the Borel model of the equivariant cohomology group $H^*_{G}(X)$ is given by
\begin{align*}
    H^*_{G}(X) = H^*(EG \times_G X).
\end{align*}
Here $EG$ is a contractible space on which $G$ acts freely, with $BG = EG/G$ the classifying space of $G$. The spaces $EG$ and $BG$ are defined uniquely up to homotopy. The homotopy quotient $EG \times_G X$ is the quotient $(EG \times X)/G$ where $G$ acts on $(e, x) \in EG \times X$ by $g \cdot (e, x) = (eg, g\inv x)$. In the case that $G$ acts freely on $X$, $H^*_{G}(X) = H^*(X/G)$ and if $G$ acts trivially on $X$, then $H^*_{G}(X) = H^*(BG \times X)$. For more detail, see \cite{Bott-EqCohom}.

For our purposes, $G$ we will always be a product of circles. Let $G = S^1$ so that $BG = \CP^\infty$ and $EG = S^\infty$. Recall that the cohomology ring of $\CP^\infty$ with coefficients in $\Z$ is $\Z[x]$ with $x$ having degree 2 \cite{Hatcher}. Suppose that $\pi : E \to B$ is a fibre-bundle with a fibre-preserving left $G$-action on $E$. We treat $B$ as a $G$-space with trivial action. Then $\HS^*(B) = H^*(\CP^\infty \times B) = H^*(B)[x]$ and $\HS^*(E)$ is a module over $H^*(B)[x]$. The module structure is defined by the pullback $\pi^*$. For $\alpha \in H^*(B)[x]$ and $\theta \in \HS^*(E)$, we will often abuse notation by writing $\alpha \cdot \theta$ to denote $\pi^*(\alpha) \wedge \theta$. The pullback of an equivariant bundle map $f : E' \to E$ is a $H^*(B)[x]$-module morphism $f^* : \HS^*(E) \to \HS^*(E')$.

Now let $\pi : W \to B$ be a rank $k$ oriented real vector bundle. Recall that $S_W$ denotes the unit sphere bundle of $\R \oplus W$ with $B \subset S$ the image of the section at infinity. In ordinary cohomology, the Thom isomorphism $H^*(B) \to H^{*+k}(S_W, B)$ is given by $\alpha \mapsto \pi^*(\alpha) \smile \Phi$, where $\Phi \in H^k(S_W, B)$ is the Thom class of $W$ \cite[$\S 6$]{Bott-Tu}. The Thom class is defined by the property that the restriction of $\Phi$ to each fibre $(S_W)_b$ is the positive generator of $H^k((S_W)_b) = \Z$. The given orientation of $W$ determines $\Phi$ uniquely. The Euler class $e(W) \in H^k(B)$ is the pullback of $\Phi$ by the zero section $\zeta : B \to S_W$. The `wrong-way' map $\pi_* : H^{*+k}(S_W, B) \to H^*(B)$ is the inverse of the Thom isomorphism, which is equivalent to integration over the fibre when using real coefficients.

The same results translate directly to equivariant cohomology. Let $W_G = EG \times W$ and $B_G = EG \times_G B$ so that $W_G \to B_G$ is an oriented equivariant vector bundle with fibre $W$. Now $H_G^*(S_W, B) = H^*(S_{W_G}, B_{W_G})$ and the equivariant Thom class $\tau_W \in H_G^k(S_W, B)$ is defined to be the ordinary Thom class of $W_G \to B_G$. As before, the equivariant Thom isomorphism $H_G^*(B) \to H_G^{*+k}(S_W, B)$ is induced by taking the cup product with $\tau_W$. Similarly, the equivariant Euler class $e(W) \in H_G^k(B)$ is just the pullback of $\tau_W$ by the zero section. Once again the `wrong-way' map $\pi_* : H_G^{*+k}(S_W, B) \to H_G^*(B)$ is the inverse of the Thom isomorphism. More generally, if $f : M \to N$ is an equivariant map of oriented $G$-manifolds with $k = \dim M - \dim N$, then $f_* : H_G^*(M) \to H_G^{*-k}(N)$ can be defined by factoring $f$ as an inclusion followed by a projection \cite{AtiyahEqCohom}.

\subsection{Cohomology of sphere bundles}

Let $V \to B$ be a complex vector bundle of rank $a$ and $U \to B$ be a real oriented vector bundle of rank $b$. Recall that $S_{V, U} \to B$ denotes the unit sphere bundle in $\R \oplus V \oplus U$ and $B \subset S_{V, U}$ is the image of the section at infinity. The inclusions $V \subset V \oplus U$ and $U \subset V \oplus U$ extend to inclusions $S_V \subset S_{V,U}$ and $S_U \subset S_{V,U}$. The orientation of $U$ induces an orientation of $S_{V,U}$ since $V$ has a natural orientation determined by its complex structure. As discussed above, the Thom isomorphism gives $\HS^*(S_{V,U}, B)$ the structure of a free rank one $\HS^*(B)$-module generated by the equivariant Thom class $\tau_{V,U} \in \HS^*(S_{V,U}, B)$.

Let $e(V \oplus U) \in \HS^{2a + b}(B)$ denote the equivariant Euler class, which is the pullback of $\tau_{V,U} \in \HS^{2a + b}(S_{V, U}, B)$ by the zero section. In $\HS^{*}(B)$, $e(V \oplus U)$ factors as $e(V \oplus U) = e(V) \cdot e(U)$. Here $e(V)$ is the equivariant Euler class of $V$ and $e(U)$ is just the ordinary Euler class of $U$ since $S^1$ acts trivially on $U$. Let $\alpha_i(V) \in H^{2}(B)$ be the ordinary chern roots of $V$. Since $S^1$ acts by scalar multiplication on $V$, the equivariant Chern roots of $V$ are $\alpha_i(V) + x$ for $1 \leq i \leq a$, where $x$ is the generator of $\HS^*(B) = H^*(B)[x]$. By the splitting principle \cite{Milnor-CharClasses}, it follows that the equivariant Euler class is given by
\begin{align}\label{E:EquivEulerClass}
    e(V) &= x^a + c_1(V)x^{a-1} + c_2(V)x^{a-2} + ... + c_a(V).
\end{align}
Notice that this is a monic polynomial and hence is not a zero divisor in $\HS^*(B)$ so long as $a \geq 1$. If $a = 0$, we of course have $e(V) = 0$.

Let $\pi_{U} : S_U \to B$ denote the projection onto $B$. The pullback bundle $\pi_{U}^*(V) \to S_U$ is a subbundle of $S_{V, U}$ which can be identified with the normal bundle of $S_U \subset S_{V,U}$. 
 
\begin{lemma}\label{L:ThomClassPullback}
    Let $i_V : S_V \to S_{V,U}$ and $i_U : S_U \to S_{V,U}$ denote inclusion maps with pullbacks $i_V^* : \HS^*(S_{V,U}, B) \to \HS^*(S_V, B)$ and $i_U^* : \HS^*(S_{V,U}, B) \to \HS^*(S_U, B)$. Then 
    \begin{align*}
        i_V^*(\tau_{V,U}) &=  \tau_V \cdot e(U)\\
        i_U^*(\tau_{V,U}) &=  e(V) \cdot \tau_U.
    \end{align*}
\end{lemma}
\begin{proof}
    Let $p_1 : (S_V \times_B S_U, B \times_B S_U) \to (S_V, B)$ and $p_2 : (S_V \times_B S_U, S_V \times_B B) \to (S_U, B)$ denote projection maps onto the first and second factor respectively. There is an external cup product defined by 
    \begin{align*}
        \smile : \HS^*(S_V, B) \otimes_{\HS^*(B)}& \HS^*(S_U, B) \to \HS^*(S_{V, U}, B)\\
        \alpha \smile \beta &= p_1^*(\alpha) \cdot p_2^*(\beta).
    \end{align*}
    The multiplication on the right hand side is the ordinary cup product in $\HS^*(S_{V, U}, B)$. Since $\HS^*(S_V, B)$ and $\HS^*(S_U, B)$ are both one-dimensional free $\HS^*(B)$-modules, this external cup product is an isomorphism \cite[Theorem 3.18]{Hatcher}. Under this isomorphism, we can identify $\tau_{V, U} = \tau_V \smile \tau_U$.

    We prove the second formula. The map $i_U$ is homotopy equivalent to the zero section of the bundle $\pi_U^*(V) \to S_U$, hence it follows that 
    \begin{align*}
        i_U^*(\tau_{V,U}) &= i_U^*(\tau_{V} \smile \tau_U)\\
        &= i_U^*(p_1^*(\tau_{V})) \cdot i_U^*(p_2^* \tau_U)\\
        &= e(V) \cdot \tau_U.
    \end{align*}
\end{proof}
Let $\delta : \HS^{*-1}(S_U, B) \to \HS^*(S_{V,U}, S_U)$ be the connecting morphism in the long exact sequence of the triple $(S_{V,U}, S_U, B)$. Recall that the ordinary cohomology of the projective bundle $\P(V) \to B$ is given by $H^*(\P(V)) = H^*(B)[x]/\<e(V)\>$ where $e(V)$ is the equivariant Euler class given by equation $(\ref{E:EquivEulerClass})$.
\begin{lemma}[\cite{BaragliaKonnoBFandSW} Proposition 3.3]\label{L:HSuvCohom}
    Suppose that $a = \rank(V) \geq 1$. Then the cohomology ring $\HS^*(S_{V,U}, S_U)$ is a free rank one $H^*(\P(V))$-module generated by $\delta \tau_U \in \HS^{b+1}(S_{V,U}, S_U)$.
\end{lemma}
\begin{proof}
    The long exact sequence of the triple $(S_{V,U}, S_U, B)$ is given by 
    \[
    \begin{tikzcd}
        ... \arrow{r} & \HS^*(S_{V,U}, B) \arrow{r}{i_U^*}\arrow[equal]{d} & \HS^*(S_U, B) \arrow{r}{\delta}\arrow[equal]{d} & \HS^{*+1}(S_{V,U}, S_U) \arrow{r} & ...\\
        & \HS^*(B)\cdot \tau_{V,U} & \HS^*(B) \cdot \tau_U
    \end{tikzcd}
    \]
    From Lemma \ref{L:ThomClassPullback}, $i_U^*(\tau_{V,U}) = e(V) \tau_U$. Since $e(V)$ is not a zero divisor in $\HS^*(B)$, the map $i_U^* : \HS^*(B)\cdot \tau_{V,U} \to \HS^*(B) \cdot \tau_U$ is injective with image $\HS^*(B) \cdot (e(V)\tau_U)$. Exactness implies that $\delta$ is surjective with kernel $\HS^*(B) \cdot (e(V)\tau_U)$. Thus $\delta$ induces an isomorphism 
    \begin{align*}
        \delta : \left(\HS^*(B)/\<e(V)\>\right) \cdot \tau_U &\to H^{*+1}(S_{V,U}, S_U)\\
        \alpha \tau_U &\mapsto \alpha \delta\tau_U.
    \end{align*}
    Recall that $\HS^*(B) = H^*(B)[x]$, hence $\HS^*(B)/\<e(V)\> = H^*(\P(V))$ as required.
\end{proof}

Now let $W \subset U$ be a proper subbundle of $U$ with dimension $b'$. A crucial ingredient in the above calculation of $\HS^*(S_{V,U}, S_U)$ is the fact that $e(V)$ is not a zero divisor, hence $i_U^*$ is injective. The group $\HS^*(S_{V, U}, S_W)$ is less straightforward to calculate since $e(V \oplus U/W)$ might be a zero divisor.

Suppose there exists a section $\phi : B \to U$ which is disjoint from $W$. Then one can define a refined Thom class $\ttau^\phi_{V,U}$ in the following manner. Let $N$ be an open tubular neighbourhood of the image $\phi(B) \subset U$, which is diffeomorphic to the total space of $U$. Since $\phi(B)$ and $W$ are disjoint closed spaces, it can be assumed that $N$ is disjoint from $W$. Define a homeomorphism $p^\phi_U : N \to U$ and extend it to a continuous map $p^\phi_U : S_U \to S_U$ by collapsing $S_U - N$ to infinity fiberwise. Since $S_W \subset (S_U - N)$, $p^\phi_U$ defines a map of pairs
\begin{align*}
    p^\phi_U : (S_U, S_W) \to (S_U, B).
\end{align*}
\begin{definition}\label{D:RefinedThomClass}
    Given a chamber $\phi : B \to (U - W)$, the refined Thom class $\ttau_{V,U}^\phi \in \HS^*(S_{V,U}, S_W)$ is defined by the formula
    \begin{align*}
        \ttau_{V,U}^\phi = (p^\phi_U)^*\tau_{V,U}.
    \end{align*} 
\end{definition}
This definition is of course independent of the homotopy class of $[\phi]$. Further, under the restriction map $i_W^* : \HS^{2a + b}(S_{V,U}, S_W) \to \HS^{2a + b}(S_{V,U}, B)$ the refined Thom class $\ttau_{V,U}^\phi$ is sent to the ordinary Thom class $\tau_{V,U}$. This is due to the fact that $(p^\phi_U)^*$ is a degree one map, hence $i_W^*(\ttau_{V,U}^\phi)$ restricts to the positive generator of $\HS^*((S_{V,U})_b, \infty_b)$ for each $b \in B$.
\begin{lemma}[\cite{BaragliaKonnoBFandSW} Proposition 3.2]\label{L:SUV-WCohom}
    Suppose that $W \subset U$ is a proper subbundle and $\phi : B \to U$ is a section which is disjoint from $W$. Let $\delta : \HS^*(S_W, B) \to \HS^*(S_{V,U}, S_W)$ be the connecting morphism in the long exact sequence of the triple $(S_{V,U}, S_W, B)$. Then $\HS^*(S_{V,U}, S_W)$ is a free rank two $\HS^*(B)$-module generated by $\ttau_{V,U}^\phi$ and $\delta \tau_W$.
\end{lemma}
\begin{proof}
    This follows immediately from the following long exact sequence.
    \[
    \begin{tikzcd}
        ... \arrow{r} & \HS^{*-1}(S_W, B) \arrow{r}{\delta} & \HS^*(S_{V,U}, S_W) \arrow{r}{i_W^*} & \HS^*(S_{V,U}, B) \arrow{r} & ...
    \end{tikzcd}
    \]
    Since $i^*_W \ttau_{V,U}^\phi = \tau_{V,U}$, the map $i^*_W$ is surjective in all degrees and therefore the above long exact sequence splits into short exact sequences.
    \[
    \begin{tikzcd}
        0 \arrow{r} & \HS^{*-1}(S_W, B) \arrow[hook]{r}{\delta} & \HS^*(S_{V,U}, S_W) \arrow[two heads]{r}{i_W^*} & \HS^*(S_{V,U}, B) \arrow{r} & 0
    \end{tikzcd}
    \]
    The choice of lift $\ttau_{V,U}^\phi$ of $\tau_{V,U}$ determines a splitting the result follows.
\end{proof}
Let $\tU$ denote the open unit disk bundle in $\pi_{U}^*(V)$, which is an equivariant tubular neighbourhood of $S_U$ in $S_{V,U}$. Note that $S^1$ acts freely on $\pi_U^*(V)$. Define $\tY = S_{V,U} - \tU$, which is a compact manifold with boundary $\p \tY = S(\pi_{U}^*(V))$, the unit sphere bundle of $\pi_{U}^*(V)$. The $S^1$-action on $\tY$ is free, hence the quotient $Y = \tY/S^1$ is a compact smooth manifold with boundary $\p Y = S(\pi_{U}^*(V))/S^1 = \P(\pi_{U}^*(V))$, the projective bundle of $\pi_{U}^*(V)$. The Chern class of the principle $S^1$-bundle $\tY \to Y$ is the image of $x \in H^*(B)[x]$ under the pullback $\HS^*(B) \to \HS^*(\tY) \cong H^*(Y)$. The restriction map $H^*(Y) \to H^*(\p Y)$ sends $x$ to the Chern class of $\Ocal_{\pi_{U}^*(V)}(1) \to \P(\pi_{U}^*(V))$.

It is clear by excision that, as $\HS^*(B)$-modules, 
\begin{align*}
    \HS^*(S_{V,U}, S_U) &\cong \HS^*(\tY, \p \tY)\\
    &= H^*(Y, \p Y).
\end{align*} 
The last equality follows since $S^1$ acts freely on $Y$. To make an explicit identification of $H^*(Y, \p Y)$ and $H^*(S_{V,U}, S_U)$, first notice that 
\begin{align*}
    \p \tY &= S(\pi_U^*(V))\\
    &= \pi_U^*(S(V))\\
    &= S_U \times_B S(V).
\end{align*}
Thus the below diagram is a pullback.
\[
\begin{tikzcd}
    \p \tY \arrow{r}{\tp_2} \arrow{d}{\tp_1} & S(V) \arrow{d}{\pi_V}\\
    S_U \arrow{r}{\pi_U} & B
\end{tikzcd}
\]
Let $\tdelta : \HS^{* - 1}(\p \tY) \to \HS^*(\tY, \p \tY)$ denote the connecting map in the long exact sequence of the pair $(\tY, \p \tY)$. The following result is adapted from \cite[Proposition 3.5]{BaragliaKonnoBFandSW}
\begin{lemma}\label{L:GammaDef}
    There is an isomorphism $\gamma : \HS^*(S_{V,U}, S_U) \to H^*(Y, \p Y)$ of $\HS^*(B)$-modules which makes the following diagram commute.
    \begin{equation}
        \begin{tikzcd}\label{D:Gamma-p1}
            \HS^{*-1}(S_U, B) \arrow{r}{\delta} \arrow{d}{\tp_1^*} & \HS^*(S_{V,U}, S_U) \arrow{d}{\gamma}\\
            H^{*-1}(\p Y) \arrow{r}{\tdelta} & H^*(Y, \p Y)
        \end{tikzcd}
    \end{equation}
\end{lemma}
\begin{proof}
    Let $i : (\tY, \p \tY, \emptyset) \to (S_{V,U}, \tU, B)$ be the inclusion of triples, which induces the following commutative diagram between long exact sequences 
    \[
    \begin{tikzcd}
        \dots \arrow{r} & \HS^{*-1}(\tU, B) \arrow{r}{\delta} \arrow{d}{i^*} & \HS^*(S_{V,U}, \tU) \arrow{d}{i^*}[swap]{\cong} \arrow{r} & \HS^*(S_{V,U}, B) \arrow{r}\arrow{d}{i^*} & \dots\\
        \dots \arrow{r} & \HS^{*-1}(\p \tY) \arrow{r}{\tdelta} & \HS^*(\tY, \p \tY) \arrow{r} & \HS^*(\tY) \arrow{r} & \dots
    \end{tikzcd}
    \]
    It follows from excision that $i^* : \HS(S_{V,U}, \tU) \to \HS(\tY, \p \tY)$ is an isomorphism. Similarly, let $j : (S_{V,U}, S_U, B) \to (S_{V,U}, \tU, B)$ be the inclusion of triples, which induces the following commutative diagram between long exact sequences 
    \[
    \begin{tikzcd}
        \dots \arrow{r} & \HS^{*-1}(S_U, B) \arrow{r}{\delta}  & \HS^*(S_{V,U}, S_U)  \arrow{r} & \HS^*(S_{V,U}, B) \arrow{r} & \dots\\
        \dots \arrow{r} & \HS^{*-1}(\tU, B) \arrow{r}{\delta} \arrow{u}{j^*}[swap]{\cong}  & \HS^*(S_{V,U}, \tU) \arrow{r} \arrow{u}{j^*}[swap]{\cong} & \HS^*(S_{V,U}, B) \arrow{u}{j^*}[swap]{\cong} \arrow{r} & \dots
    \end{tikzcd}
    \]
    Since $j$ is a homotopy equivalence, all the vertical maps are isomorphisms. Since $\tU = S_U \times_B D(V)$, the projection map $p : \tU \to S_U$ is a homotopy inverse for $j$. Combining the two diagram gives 
    \[
    \begin{tikzcd}
        \HS^{*-1}(S_U, B) \arrow{r}{\delta} \arrow[bend right=60]{d}[swap]{p^*}  & \HS^*(S_{V,U}, S_U) \arrow[bend left=60]{dd}{\gamma}\\
        \HS^{*-1}(\tU, B) \arrow{r}{\delta} \arrow{u}{j^*}[swap]{\cong} \arrow{d}{i^*} & \HS^*(S_{V,U}, \tU) \arrow{u}{j^*}[swap]{\cong} \arrow{d}{i^*}[swap]{\cong}\\
        \HS^{*-1}(\p \tY) \arrow{r}{\tdelta} & \HS^*(\tY, \p \tY)
    \end{tikzcd}
    \]
    Let $\gamma = i^* \circ (j^*)\inv$, which is an isomorphism. Further $i^* \circ p^* = (p \circ i)^* = \tp_1^*$, hence diagram (\ref{D:Gamma-p1}) commutes.
\end{proof}
The isomorphism $\gamma$ naturally identifies $H^*(Y, \p Y)$ with $\HS(S_{V,U}, S_U)$ and thus identifies $H^*(Y, \p Y)$ as a free rank one $H^*(\P(V))$-module generate by $\gamma(\delta \tau_U)$. Diagram (\ref{D:Gamma-p1}) shows that calculations on $H^*(Y, \p Y)$ simplify to calculations in $H^*(\P(V))$ in the sense of the following result.
\begin{proposition}\label{P:PiYReduction}
    Let $\alpha \in H^*(Y, \p Y)$ and write $\alpha = \eta \cdot \gamma(\delta \tau_U)$ for some $\eta \in H^*(\P(V))$. Then
    \begin{align*}
        (\pi_Y)_*(\alpha) &= (\pi_{\P(V)})_*(\eta).
    \end{align*}
\end{proposition}
\begin{proof}
    From Lemma \ref{L:GammaDef} we have 
    \begin{align*}
        (\pi_Y)_*(\alpha) &= (\pi_Y)_*(\gamma(\eta\cdot \delta \tau_U))\\
        &= (\pi_Y)_*(\gamma(\delta (\eta \cdot \tau_U)))\\
        &= (\pi_Y)_*(\tdelta(\tp_1^* (\eta \cdot \tau_U))).
    \end{align*}
    Let $\iota : \p Y \to Y$ be the inclusion map. It was shown in \cite[Lemma 3.4]{BaragliaKonnoBFandSW} by direct calculation that $\tdelta = \iota_*$, hence
    \begin{align*}
        (\pi_Y)_*(\alpha) &= (\pi_Y)_* \left(\iota_*(\tp_1^* (\eta \cdot \tau_U))\right)\\
        &= (\pi_{\p Y})_* (\eta \cdot \tp_1^*\tau_U).
    \end{align*}
    Note that $(\pi_{\p Y})_* = (\pi_Y)_* \circ \iota_*$. Recall that $\p Y = S_U \times_B \P(V)$ and consider the following diagram.
    \[
    \begin{tikzcd}
        \p Y \arrow{r}{p_1} \arrow{d}{p_2} & S_U \arrow{d}{\pi_U} \\
        \P(V) \arrow{r}{\pi_{\P(V)}} & B
    \end{tikzcd}
    \]
    The bundle $p_2 : S_U \times_B \P(V) \to \P(V)$ is a sphere bundle over $\P(V)$ which is the pullback of $S_U$ by $\pi_{\P(V)}$. We can identify $\P(V) \subset \p Y$ as the image of the section at infinity. Viewing $\tp_1^*(\tau_U)$ as in element of $H^*(\p Y, \P(V)) = \HS^*(\p \tY, S(V))$ we see that $\tp_1^*(\tau_U)$ is the Thom class of $S_U \times_B \P(V) \to \P(V)$. Hence $(p_2)_* (\tp_1^*(\tau_U)) = 1$. Writing $\pi_{\p Y} = \pi_{\P(V)} \circ p_2$ and applying the projection formula we have 
    \begin{align*}
        (\pi_Y)_*(\alpha) &= (\pi_{\P(V)})_* \left((p_2)_* (\eta \cdot \tp_1^*\tau_U)\right)\\
        &= (\pi_{\P(V)})_*(\eta).
    \end{align*}
\end{proof}
The Segre classes $s_j(V) \in H^{2j}(B)$ of $V \to B$ are defined by the relation $s(V)c(V) = 1$ where $s(V) = 1 + s_1(V) + ... + s_a(V)$ is the total Segre class of $V$ and $c(V) = 1 + c_1(V) + ... + c_a(V)$ is the total Chern class of $V$. It can be shown by direct calculation that
\begin{align*}
    (\pi_{\P(V)})_*(x^j) = 
    \begin{cases}
        0 & \text{if $j < a-1$}\\
        s_{j-(a-1)}(V) & \text{if $j \geq a - 1$}
    \end{cases}
\end{align*}
By Proposition \ref{P:PiYReduction}, this data can be used to calculate $(\pi_Y)_*(\alpha)$ for any $\alpha \in H^*(Y, \p Y)$.

\subsection{Cohomological formula for $\SWfphi_m$}

Let $f : S_{V', U'} \to S_{V, U}$ be a finite dimensional monopole map with $D = V' - V \in K^0(B)$ and $H^+ = U - U' \in KO^0(B)$. As in Definition \ref{D:FDMonopoleMap}, we can assume that $U = U' \oplus H^+$ with $f|_{U'}$ the inclusion and $H^+$ a genuine vector bundle. Let $\phi : B \to (U - U')$ be a representative of a chamber. Let $\ttau^\phi_{V,U} \in H^{2a+b}(S_{V,U}, S_{U'})$ be the refined Thom class described in Definition \ref{D:RefinedThomClass}. Since $f|_{U'}$ is an inclusion, we have $f^* : H^*(S_{V, U}, S_{U'}) \to H^*(S_{V', U'}, S_{U'})$. One of the primary results of \cite{BaragliaKonnoBFandSW} is the following theorem 
\begin{theorem}\cite[Theorem 3.6]{BaragliaKonnoBFandSW}
    For each $m \geq 0$, we have 
    \begin{align}
        \SWfphi_m = (\pi_Y)_*(x^m \gamma(f^*\phi_*(1))).
    \end{align}
\end{theorem}
Notice the differences in this formula compared to Definition \ref{D:FDSWFormula}. The above formula makes no reference to the moduli space $\Mcal_\phi$ or the line bundle $\Lcal$. Furthermore, this definition of $\SWfphi_m$ clearly only depends on $f$ and $\phi$ up to homotopy. Moreover, since $f^*\phi_*(1) \in H^*(S_{V', U'}, S_{U'})$, we can write $f^*\phi_*(1) = \eta^\phi \delta\tau_U$ for some $\eta^\phi \in H^*(\P(V))$. Then by Proposition \ref{P:PiYReduction} we have, for any $\alpha \in H^*(\P(V))$,
\begin{align}\label{E:SWFormulaEta}
    \SWfphi_m(\alpha) = (\pi_{\P(V)})_*(x^m \eta^\phi).
\end{align}
We extend this notation to a map $\SWfphi : H^*(\P(V)) \to H^*(B)$ defined for any $\alpha \in H^*(\P(V))$ by
\begin{align}
    \SWfphi(\alpha) &= (\pi_{\P(V)})_*(\alpha \eta^\phi).
\end{align}
Of course, the $m$th Seiberg-Witten invariant is given by $\SWfphi_m = \SWfphi(x^m)$.

\section{Localised formula}\label{S:Localisation}

To derive a connected sum formula for $\SWfphi_m$, equation $(\ref{E:SWFormulaEta})$ illustrates that we can focus our attention on the map $(\pi_{\P(V)})_*$. This will be accomplished using the Atiyah-Bott localisation theorem \cite{AtiyahEqCohom} in equivariant cohomology.

Given a ring $R$ and a multiplicative subset $S \subset R$, the localisation of $R$ with respect to $S$ is 
\begin{align*}
    S\inv R = \left\lbrace\frac{r}{s} \mid r \in R, s \in S\right\rbrace/\sim
\end{align*}
where $\frac{r_1}{s_1} \sim \frac{r_2}{s_2}$ if and only if there exists some $t \in S$ such that 
\begin{align*}
    t(r_1s_2 - r_2s_1) = 0.
\end{align*}
The localisation $S\inv R$ is a ring by adding and multiplying fractions in the usual way. There is a natural map $l_S : R \to S\inv R$ defined by 
\begin{align}\label{E:LocInc}
    l_S(r) &= \frac{r}{1}.
\end{align}
This map has the property that $l_S(S)$ is contained in the group of units of $S\inv R$. If $S$ contains no zero divisors, then $l_S$ is injective and $R$ can be identified as a subring of $S\inv R$. Alternatively, $S\inv R$ can be defined by the universal property that for any ring map $f : R \to T$ with $f(S)$ contained in the group of units of $T$, there is a unique extension $\tilde{f} : S\inv R \to T$ such that the following diagram commutes.
\[
\begin{tikzcd}
    S\inv R \arrow[dashed]{rr}{\exists ! \tilde{f}} && T\\
    & R \arrow{ul}{l_S} \arrow[swap]{ur}{f}
\end{tikzcd}
\]
Let $f : R_1 \to R_2$ be a morphism of rings with multiplicative subsets $S_1 \subset R_1, S_2 \subset R_2$. The universal property of localisation implies that $l_{S_2} \circ f$ extends uniquely to a morphism $\tilde{f} : S_1\inv R_1 \to S_2\inv R_2$ so long as $f(S_1) \subset S_2$. The extension makes the following diagram commute.
\begin{equation}\label{E:LocExt}
    \begin{tikzcd}
        S_1\inv R_1 \arrow[dashed]{r}{\exists ! \tilde{f}} & S_2\inv R_2\\
        R_1 \arrow{r}{f}\arrow{u}{l_{S_1}} & R_2 \arrow[swap]{u}{l_{S_2}}
    \end{tikzcd}
\end{equation}
We will often abuse notation by just writing $f : S_1\inv R_1 \to S_2\inv R_2$ to denote the extension $\tilde{f}$. For more detail on localised rings, see \cite{Eisenbud1995}.

As usual let $V_1, V_2$ be complex vector bundles of complex rank $a_1$ and $a_2$ respectively. Set $V = V_1 \oplus V_2$ with $a = a_1 + a_2$. Define an abnormal action on $V$ by letting $S^1$ act by scalar multiplication on the $V_1$ component and trivially on the $V_2$ component. Notice that this is different to the usual action of scalar multiplication on both components of $V$. Let $\pi : \P(V) \to B$ be the complex projective bundle of $V$ and extend this $S^1$-action to $\P(V)$. Let $\iota_i : \P(V_i) \to \P(V)$ be inclusions for $i \in \{1,2\}$ and note that the subset of fixed points of $\P(V)$ is exactly $\P(V_1) \cup \P(V_2)$.

%Our goal is to use the localisation theorem in equivariant cohomology \cite[Theorem 3.5]{AtiyahEqCohom} to explicitly described the map $\pi_* : \HS^{*+2a}(\P(V)) \to \HS^*(B)$. 
Recall that the $S^1$-equivariant cohomology of a point is $\HS(*) = H^*(\CP^\infty)$ and pullback by the terminal map $\P(V) \to \{*\}$ endows $\HS^{*+2a}(\P(V))$ with the structure of a $\Z[y]$-module. Here $y$ is the pullback of the generator of $H^*(\CP^\infty) = \Z[y]$, which has degree 2. Let $S \subset \HS^{*}(\P(V))$ be the subset $S = \{1, y, y^2, ...\}$ containing all non-negative powers of $y$. The map 
\begin{align*}
    \iota_1^* \oplus \iota_2^* : \HS^*(\P(V)) \to \HS^*(\P(V_1)) \oplus \HS^*(\P(V_2))
\end{align*}
extends to a map 
\begin{align*}
    \iota_1^* \oplus \iota_2^* : S\inv \HS^*(\P(V)) \to S\inv \HS^*(\P(V_1)) \oplus S\inv \HS^*(\P(V_2))
\end{align*}
which satisfies the commutativity relation outlined in Diagram (\ref{E:LocExt}). The localisation theorem as described in \cite[Proposition 4.3]{AtiyahEqCohom} states that $\iota_1^* \oplus \iota_2^*$ induces an isomorphism between the localised cohomology groups. We will derive an explicit formula for its inverse.

Let $N_i$ be the normal bundle of $\P(V_i)$ in $\P(V)$. Denote by $\Ocal(-1) \to \P(V)$ the tautological line bundle with fibre at a line $l_b \in \P(V)_b$ the one-dimensional subspace $l_b \subset \pi^*(V)_b$. We will write $\Ocal_i(-1) \to \P(V_i)$ to denote the tautological line over $\P(V_i)$, noting that $\Ocal_i(-1) = \iota_i^* \Ocal(-1)$. Let $\Ocal(1) = \Hom(\Ocal(-1), \C)$ denote the dual of $\Ocal(-1)$. We view $V_j$ as a vector bundle over $\P(V_i)$ through pull back by $\pi_i : \P(V_i) \to B$.
\begin{lemma}\label{L:NiCalc}
    There are equivariant isomorphisms $N_1 \cong (V_2 \otimes \Ocal_1(1))_{-1}$ and $N_2 \cong (V_1 \otimes \Ocal_2(1))_1$, where the subscripts denote the weight of the $S^1$-action.
\end{lemma}
\begin{proof}
    Let $T(\P(V)/B)$ denote the vertical tangent bundle of $\pi : \P(V) \to B$. Consider the following short exact sequence of vector bundles over $\P(V)$, defined by inclusion and projection maps.
    \begin{align*}
        0 \to \Ocal(-1) \to V \to V/\Ocal(-1) \to 0
    \end{align*}
    Tensoring with $\Ocal(1)$ produces the following exact sequence of vector bundles over $\P(V)$, known as the Euler sequence.
    \begin{align*}
        0 \to \C \to V \otimes \Ocal(1) \to T(\P(V)/B) \to 0
    \end{align*}
    Pulling back by $\iota_i : \P(V_i) \to \P(V)$ gives an exact sequence of vector bundles over $\P(V_i)$. Combining this with the Euler sequence for $V_i$, we obtain the following commutative diagram of vector bundles over $\P(V_i)$.
    \[
    \begin{tikzcd}
                    & 0 \arrow{d}           & 0 \arrow{d}                               & 0 \arrow{d}\\
        0 \arrow{r} & \C \arrow{r}\arrow{d} & V_i \otimes \Ocal_i(1) \arrow{r}\arrow{d} & T(\P(V_i)/B) \arrow{r}\arrow{d} & 0\\
        0 \arrow{r} & \C \arrow{r}\arrow{d} & V \otimes \Ocal_i(1) \arrow{r}\arrow{d} & \iota_i^*T(\P(V)/B) \arrow{r}\arrow{d} & 0\\
                   & 0 & V/V_i \otimes \Ocal_i(1) \arrow[dashed]{r}{f}\arrow{d} & N_i \arrow{r}\arrow{d} & 0\\
                    &                                                 & 0 & 0 &
    \end{tikzcd}
    \]
    The columns and first two rows of this diagram are exact. An elementary diagram chase shows that there is a well defined map $f : V/V_i \otimes \Ocal_i(1) \to N_i$ represented by the dashed arrow. The nine lemma proves that $f$ is an isomorphism. Identifying $V/V_1 \cong V_2$ and $V/V_2 \cong V_1$ shows that $N_1 \cong (V_2 \otimes \Ocal_1(1))$ and $N_2 \cong (V_1 \otimes \Ocal_2(1))$ on the level of ordinary vector bundles.

    To show that the weight of the $S^1$-action on $N_1$ is $-1$, fix $l_b \in \P(V_1)$ for some $b \in B$ and local trivialisations $V_1|_U \cong \C^{a_1}$, $V_2|_U \cong \C^{a_2}$ for some open $U \subset \P(V_1)$ containing $l_b$. We can assume that $l_b = [1,0,...,0]$ in the homogenous coordinates associated to this trivialisation. We can also assume that $V|_U \cong \C^{a_1} \oplus \C^{a_2}$ is a trivialisation with $V_1|_U = \C^{a_1} \oplus 0$ and $V_2|_U = 0 \oplus \C^{a_2}$. A vector $v \in (N_1)_{l_b}$ in a fibre of $N_1$ can be represented as the tangent vector to a curve $\gamma_t$ transverse to $\P(V_1)|_U \subset  \P(V)|_U$, given by.
    \begin{align*}
        \gamma_t &= [1,0,...,0] + t[0,...,0,v_1,...,v_n]\\
        &= [1,0,...,0,tv_1,...,tv_n].
    \end{align*}
    Recall that $e^{i\theta} \cdot [w,z] = [e^{i\theta}w, z]$ for $w \in V_1|_U$, $z \in V_2|_U$ with at least one of them non-zero. Acting on $\gamma_t$ gives 
    \begin{align*}
        e^{i\theta} \cdot \gamma_t &= e^{i\theta} \cdot [1,0,...,0,tv_1,...,tv_n]\\
        &= [e^{i\theta},0,...,0,tv_1,...,tv_n]\\
        &= [1,0,...,0,te^{-i\theta}v_1,...,te^{-i\theta}v_n]\\
        &= [1,0,...,0] + t[0,...,0,e^{-i\theta}v_1,...,e^{-i\theta}v_n].
    \end{align*}
    Hence $S^1$ acts on $N_1$ with weight $-1$. A similar argument shows that $S^1$ acts on $N_2$ with weight $+1$.
\end{proof}
Since $S^1$ acts trivially on $\P(V_i)$, we have that $\HS^*(\P(V_i)) = H^*(\P(V_i))[y]$ and $S\inv \HS^*(\P(V_i)) = H^*(\P(V_i))[y, y\inv]$. That is, we can identify $\HS^*(\P(V_i))$ as a subring of $S\inv \HS^*(\P(V_i))$ using the injective map defined in (\ref{E:LocInc}). Let $e(N_i) \in \HS^{a_i}(\P(V_i))$ denote the $S^1$-equivariant Euler class of $N_i$.
\begin{lemma}\label{L:e(N_i)calc}
    The equivariant Euler class $e(N_i)$ is invertible in the localised ring $S\inv\HS^*(\P(V_i))$.
\end{lemma}
\begin{proof}
    We will prove the result for $N_2$. The $N_1$ case follows the same argument with $y$ replaced with $-y$. Let $\alpha_i$ for $0 \leq i \leq a_1$ be the Chern roots of $V_1$. Then by Lemma \ref{L:NiCalc}, the Chern roots of $N_2 = (V_1 \otimes \Ocal_2(1))_1$ are $\alpha_i(V_1) + y$. Thus by the splitting principle, the total chern class $c(N_2)$ of $N_2$ is 
    \begin{align*}
        c(N_2) &= \sum_{i=0}^{a_1} \sum_{j=0}^{i} c_{j}(V_1) y^{i - j}.
    \end{align*}
    It follows that the equivariant Euler class $e(N_2)$ is given by
    \begin{align*}
        e(N_2) &= \sum_{j=0}^{a_1} c_{j}(V_1) y^{a_1 - j}.
    \end{align*}
    We claim that the inverse of $e(N_2)$ in $S\inv\HS^{*}(\P(V_2))$ is given by 
    \begin{align*}
        e(N_2)\inv &= \sum_{k = 0}^{a_1} s_k(V_1)y^{-a_1-k}
    \end{align*}
    Here $s_j(V_1) \in H^{2j}(B)$ denotes the $j$-th Segre class and is defined by the equation $c(V_1)s(V_1) = 1$ for $s(V_1) = 1 + s_1(V_1) + s_2(V_1) + ... $ the total Segre class. Note that this sum is finite since $B$ is a finite dimensional manifold. It follows that
    \begin{align*}
        e(N_2)e(N_2)\inv &= \sum_{j=0}^{a_1}\sum_{k = 0}^{a_1}  c_{j}(V_1)s_k(V_1)y^{-(j+k)}\\
        &= \sum_{m=0}^{2a_1}\left(\sum_{j=0}^{m}  c_{j}(V_1)s_{m-j}(V_1)\right)y^{-m}.
    \end{align*}
    Since $c(V_1)s(V_1) = 1$, we must have that $\sum_{j=0}^{m}  c_{j}(V_1)s_{m-j}(V_1)$ is zero for $m > 0$ by the grading of $H^*(B)$. Therefore $e(N_2)e(N_2)\inv = 1$.
\end{proof}

\begin{lemma}\label{L:iotaCalcs}
    For any $\alpha \in \HS^*(\P(V_j))$ with $j, j' \in \{1,2\}$ distinct, we have
    \begin{align*}
        \iota^*_j((\iota_j)_*\alpha) &= e(N_j)\alpha\\
        \iota^*_j((\iota_{j'})_*\alpha) &= 0.
    \end{align*}
\end{lemma}
\begin{proof}
    Let $\pi_j : N_j \to \P(V_j)$ denote the projection with Thom class $\tau_j$. First note that since $(\pi_j)_*$ is an isomorphism, $(\iota_j)_*\alpha = \pi_j^*\alpha \cdot \tau_j$. This follows because $(\pi_j)_*(\iota_j)_*$ is the identity and $(\pi_j)_*(\pi_j^*\alpha \cdot \tau_j) = \alpha$ by the projection formula. Applying this produces the following calculation.
    \begin{align*}
        \iota^*_j((\iota_j)_*\alpha) &= \iota^*_j(\pi_j^*\alpha \cdot \tau_i)\\
        &= \alpha \iota^*_j(\tau_j)\\
        &= \alpha e(N_j).
    \end{align*}
    Note that the inclusion map $\iota_j$ is just the zero section of the normal bundle $N_j$, hence pulling back $\tau_j$ by $\iota_j$ gives $e(N_j)$. For $j,j' \in \{1,2\}$ distinct, we have 
    \begin{align*}
        \iota_j^*(\iota_{j'})_* \alpha &= \iota_j^* (\pi_{j'}^*\alpha \cdot \tau_{j'})\\
        &= (\pi_{j'} \circ \iota_j)^* \alpha \cdot \iota_j^* \tau_j\\
        &= 0.
    \end{align*}
    The last equality follows since $\pi_{j'} \circ \iota_j$ is a constant map.
\end{proof}
The above relations also hold for the extended maps $\iota^*_j : S\inv\HS^*(\P(V)) \to S\inv\HS^*(\P(V_j))$ and $(\tiota_j)_* : S\inv\HS^*(\P(V_j)) \to S\inv\HS^*(\P(V))$ by the uniqueness outlined in diagram (\ref{E:LocExt}).
\begin{lemma}\label{L:GammaEquation}
    For any $\gamma \in S\inv\HS^*(\P(V))$, we have
    \begin{align*}
        \gamma &= (\iota_1)_*(e(N_1)\inv \iota^*_1\gamma) + (\iota_2)_*(e(N_2)\inv \iota^*_2\gamma)
    \end{align*}
\end{lemma}
\begin{proof}
    Let $\beta = (\iota_1)_*(e(N_1)\inv \iota^*_1\gamma) + (\iota_2)_*(e(N_2)\inv \iota^*_2\gamma)$. Calculating using Lemma \ref{L:iotaCalcs} shows that $(\iota^*_1 + \iota^*_2)\beta = (\iota^*_1 + \iota^*_2)\gamma$, hence $\gamma = \beta$ since $\iota^*_1 + \iota^*_2$ is an isomorphism between the localised cohomology rings.
\end{proof}
Let $l_{\P(V)} : \HS^*(\P(V)) \to S\inv\HS^*(\P(V))$ denote the map sending $\alpha \in \HS^*(\P(V))$ to $l_{\P(V)}(\alpha) = \frac{\alpha}{1} \in S\inv\HS^*(\P(V))$. Note that we cannot necessarily identify $\HS^*(\P(V))$ as a subring of $S\inv \HS(\P(V))$ since $y \in S$ may be a zero divisor of $\HS^*(\P(V))$, in which case the map $\alpha \mapsto \frac{\alpha}{1}$ is not injective. However $\HS^*(B) = H^*(B)[y]$, hence $y$ is not a zero-divisor in $\HS^*(B)$ and the associated map $l_B : \HS^*(B) \to S\inv \HS^*(B)$ is the inclusion $H^*(B)[y] \subset H^*(B)[y, y\inv]$. Let $\pi_1 = \pi \circ \iota_i$ and extend $(\pi_i)_* : S\inv\HS^*(\P(V_i)) \to S\inv\HS^*(B)$ to a map between the localised cohomology rings.
\begin{theorem}[Localised Formula]\label{T:LocFormula}
    For any $\alpha \in \HS^*(\P(V))$, we have 
    \begin{align}\label{E:LocFormula}
        \pi_*(\alpha) &= (\pi_1)_*(e(N_1)\inv \iota_1^* \gamma) + (\pi_2)_*(e(N_2)\inv \iota_2^* \gamma)
    \end{align}
    where $\gamma = l_{\P(V)} (\alpha) \in S\inv\HS^*(\P(V))$. 
\end{theorem}
\begin{proof}
    From Lemma \ref{L:GammaEquation}, we have 
    \begin{align*}
        &(\pi_1)_*(e(N_1)\inv \iota_1^* \gamma) + (\pi_2)_*(e(N_2)\inv \iota_2^* \gamma)\\ 
        &= \pi_*\left((\iota_1)_*(e(N_1)\inv \iota_1^* \gamma) + (\iota_2)_*(e(N_2)\inv \iota_2^* \gamma)\right) \\
        &= \pi_* \gamma\\
        &= \pi_* l_{\P(V)} (\alpha).
    \end{align*}
    From (\ref{E:LocExt}), the following diagram is commutative.
    \begin{equation}
        \begin{tikzcd}
            S\inv \HS^*(\P(V)) \arrow{r}{\pi_*} & S\inv\HS^*(B)\\
            \HS^*(\P(V)) \arrow{r}{\pi_*}\arrow{u}{l_{\P(V)}} & \HS^*(B) \arrow[swap, hook]{u}{l_{B}}
        \end{tikzcd}
    \end{equation}
    Therefore $\pi_* l_{\P(V)} (\alpha) = l_B(\pi_*(\alpha))$. In particular, the right-hand side of (\ref{E:LocFormula}) is valued in $\HS^*(B)$ and the formula follows.
\end{proof}

\section{Seiberg-Witten connected sum formula}\label{Ch:FSW-CSF}

For $i \in \{1,2\}$, let $f_i : S_{V'_i, U'_i} \to S_{V_i, U_i}$ be two finite dimensional monopole maps where $U_i \to B$ is a real vector bundle of rank $b_i$ and $V_i \to B$ is a complex vector bundle of rank $a_i$. We assume that $U_i' \subset U_i$ with $f_i|_{U_i'}$ the inclusion. Set $U = U_1 \oplus U_2$, $V = V_1 \oplus V_2$ and define $U'$ and $V'$ similarly. Let $f = f_1 \wedge_B f_2 : S_{V',U'} \to S_{V,U}$ denote the fibrewise smash product of $f_1$ and $f_2$. Then $f|_{U'}$ is the inclusion and $f$ is a finite dimensional monopole map.

For the purpose of applying the localised formula \ref{T:LocFormula}, define a $\T = S^1 \times S^1$ action on $V = V_1 \oplus V_2$ in the following manner. The first factor $\T_1 = S^1$ acts in the usual fashion, scalar multiplication on both $V_1$ and $V_2$. The second factor $\T_2 = S^1$ acts on $V_1$ by scalar multiplication but on $V_2$ trivially. The action of $\T_2$ on $V$ is the same as the $S^1$-action defined in \S\ref{S:Localisation}. Let $\T$ act on $V'$ in the same way and extend this action to $S_{V,U}$ and $S_{V', U'}$. Once again let $\T$ act trivially on $B$ so that $\HT^*(B) = H^*(B)[x,y]$, where $x$ generates the action of the first factor $\T_1$ and $y$ generates the action of the second factor $\T_2$. By pullback, $\HT(S_{V,U})$ and $\HT(S_{V',U'})$ are modules over $H^*(B)[x,y]$.

Let $S(V) \to B$ be the unit sphere bundle in $V$. Then $\T_1$ acts freely on $S(V)$ and $\P(V) = S(V)/\T_1$. As shown in the previous section, there is an induced action of $\T_2$ on $\P(V)$ with fixed point set $\P(V_1) \cup \P(V_2)$. Moreover $\HT(S(V)) = \HTtwo(\P(V))$, which is a module over $H^*(B)[y]$.

Let $\phi_2 : B \to (U_2 - U_2')$ represent a chamber for $f_2$ with $\phi = (0,\phi_2) : B \to (U - U')$ representing a chamber for $f$. Note that not every chamber of $f$ can be represented by a map of this form, however we make the simplifying assumption that such a chamber exists. Let $\ttau^\phi_{V,U} \in \HT^*(S_{V, U}, S_U)$ be the refined Thom class induced by $\phi$. The long exact sequence of the triple $(S_{V,U}, S_U, B)$ in $\T$-equivariant cohomology is given by 
\[
\begin{tikzcd}
    ... \arrow{r} & \HT^*(S_{V,U}, B) \arrow{r}{i^*}\arrow[equal]{d} & \HT^*(S_U, B) \arrow{r}{\delta}\arrow[equal]{d} & \HT^{*+1}(S_{V,U}, S_U) \arrow{r} & ...\\
    & \HT^*(B)\cdot \tau_{V,U} & \HT^*(B) \cdot \tau_U
\end{tikzcd}
\]
Thus the proof of Lemma \ref{L:HSuvCohom} shows that $\HT^*(S_{V,U}, S_U)$ is a free rank one $\HTtwo^*(\P(V))$-module generated by $\delta \tau_U$. Recall that $\tY$ is the complement of a $\T$-invariant tubular neighbourhood of $S_U \subset S_{V,U}$ with $Y = \tY / \T_1$. As before, $\pi_Y : Y \to B$ is an oriented fibre bundle with each fibre a smooth manifold with boundary. The push forward map $(\pi_Y)_* : \HTtwo^*(Y, \p Y) \to \HTtwo^*(B)$ defines a map in $\T_2$-equivariant cohomology. Lemma \ref{L:GammaDef} and Proposition \ref{P:PiYReduction} extend to this setting and show that there is an isomorphism $\gamma : \HT^*(S_{V,U}, S_U) \to \HTtwo^*(Y, \p Y)$ which gives $\HTtwo^*(Y, \p Y)$ the structure of a free $\HTtwo^*(\P(V))$-module generated by $\gamma(\delta \tau_U)$. Further, for $\alpha \in \HTtwo^*(Y, \p Y)$ written as $\alpha = \eta \cdot \gamma(\delta \tau_U)$ with $\eta \in \HTtwo^*(\P(V))$, we have
\begin{align*}
    (\pi_Y)_*(\alpha) &= (\pi_{\P(V)})_*(\eta).
\end{align*}
Note that $(\pi_{\P(V)})_* : \HTtwo^*(\P(V)) \to \HTtwo^*(B)$ is a map in $\T_2$-equivariant cohomology and so $(\pi_{\P(V)})_*(\eta)$ is valued in $H^*(B)[y]$. Moreover, $(\pi_{\P(V)})_*$ is $H^*(B)[y]$-linear.
\begin{definition}
    The $m$th $\T_2$-generalised Seiberg-Witten invariant $\SWhatfphi_m$ of $f$ with respect to the chamber $\phi$ is given by 
    \begin{align}
        \SWhatfphi_m &= (\pi_Y)_*(x^m\gamma(f^*\ttau^\phi_{V,U}))
    \end{align}
    which is valued in $\HTtwo^*(B) = H^*(B)[y]$.
\end{definition}
We call $\SWhatfphi_m$ the $\T_2$-generalised Seiberg-Witten since it is a polynomial in $H^*(B)[y]$ and evaluating at $y = 0$ gives the ordinary Seiberg-Witten invariant. That is, 
\begin{align*}
    \SWhatfphi_m|_{y=0} &= \SWfphi_m.
\end{align*}
Also note by the discussion above that we can write $\gamma(f^*\ttau^\phi_{V,U}) = \eta^\phi \gamma(\delta \tau_U)$ for some $\eta^\phi \in \HTtwo^*(\P(V))$ such that 
\begin{align*}
    \SWhatfphi_m &= (\pi_{\P(V)})_*(x^m \eta^\phi ).
\end{align*}
In this fashion we have a map $\SWhatfphi : \HTtwo^*(\P(V)) \to \HTtwo^*(B)$ defined by 
\begin{align*}
    \SWhatfphi(\alpha) = (\pi_{\P(V)})_*(\alpha \eta^\phi).
\end{align*} 
We will use this description and the localisation formula applied to $(\pi_{\P(V)})_*$ to derive a connected sum formula for $\SWhatfphi$.
Let $\iota_i' : \P(V_i') \to \P(V')$ denote the inclusion for $i \in \{1,2\}$ with projection $\pi_i' : \P(V_i') \to B$. Also let $N_i'$ denote the normal bundle of $\P(V_i')$ in $\P(V')$, which is a $\T_2$-equivariant vector bundle. As in Theorem \ref{T:LocFormula}, let $l_{\P(V')} : \HTtwo(\P(V')) \to S\inv \HTtwo(\P(V'))$ be the canonical map sending $\alpha \mapsto \frac{\alpha}{1}$ and note that $l_{\P(V')}$ might not be injective. Denoting $l_{\P(V')}(\eta^\phi)$ by $\teta^\phi$, we have 
\begin{align}\label{E:SWLocCalc}
    \SWhatfphi_m &= (\pi_1')_*(x^m e(N_1')\inv (\iota_1')^*\teta^\phi) + (\pi_2')_*(x^m e(N_2')\inv (\iota_2')^*\teta^\phi).
\end{align}

Before applying the above formula, we first perform some cohomology calculations. The first calculation invokes the external cup product, which is described in \cite[Theorem 3.18]{Hatcher}
\begin{lemma}\label{L:RefinedThomSplit}
    Let $\phi = (0, \phi_2) : B \to (U-U')$ be a chamber with refined Thom class $\ttau^\phi_{V,U} \in \HTtwo(S_{V,U}, S_{U'})$. Then
    \begin{align}\label{E:ThomSplit}
        \ttau^\phi_{V,U} &= \tau_{V_1, U_1} \smile \ttau_{V_2, U_2}^{\phi_2}
    \end{align}
    where $\tau_{V_1, U_1} \in \HTtwo^*(S_{V_1, U_1}, B)$ is the equivariant Thom class of $S_{V_1, U_1}$ and $\ttau_{V_2, U_2}^{\phi_2} \in \HTtwo^*(S_{V_2, U_2}, S_{U_2'})$ is the refined Thom class induced by $\phi_2$.
\end{lemma}
\begin{proof}
    We define the external cup product $\tau_{V_1, U_1} \smile \ttau_{V_2, U_2}^{\phi_2} \in \HTtwo(S_{V,U}, S_{U'})$ for $\tau_{V_1, U_1} \in \HTtwo^*(S_{V_1, U_1}, B)$ and $\ttau_{V_2, U_2}^{\phi_2} \in \HTtwo^*(S_{V_2, U_2}, S_{U_2'})$, then prove that (\ref{E:ThomSplit}) holds. For ease of notation, we assume that  $V = 0$, but the general argument is identical.

    Define two projection maps 
    \begin{align*}
        p_1 &: (S_{U_1} \times_B S_{U_2}, B \times_B S_{U_2}) \to (S_{U_1}, B)\\
        p_2 &: (S_{U_1} \times_B S_{U_2}, S_{U_1} \times S_{U_2'}) \to (S_{U_2}, S_{U_2'})
    \end{align*}
    with corresponding pullbacks
    \begin{align*}
        p_1^* &: \HTtwo^*(S_{U_1}, B) \to \HTtwo^*(S_{U_1} \times_B S_{U_2}, B \times_B S_{U_2})\\
        p_2^* &: \HTtwo^*(S_{U_2}, S_{U_2'}) \to \HTtwo^*(S_{U_1} \times_B S_{U_2},  S_{U_1} \times_B S_{U_2'}).
    \end{align*}
    Note that $B \times_B S_{U_2} \cup S_{U_1} \times_B S_{U_2'} = S_{U_1} \vee_B S_{U_2} \cup S_{U_1} \times_B S_{U_2'}$ so we have 
    \begin{align*}
        p_1^* \smile p_2^* : \HTtwo^*(S_{U_1}, B) \otimes \HTtwo^*(S_{U_2}, S_{U_2'}) \to \HTtwo^*(S_{U_1} \times_B S_{U_2}, S_{U_1} \times_B S_{U_2'} \cup S_{U_1} \vee_B S_{U_2}).
    \end{align*}
    Observe that 
    \begin{align*}
        (S_{U_1} \times_B S_{U_2'} \cup S_{U_1} \vee_B S_{U_2})/(S_{U_1} \vee_B S_{U_2}) \cong (S_{U_1} \wedge_B S_{U_2'})/B.
    \end{align*}
    Applying \cite[Proposition 2.22]{Hatcher} it follows that 
    \begin{align*}
        \HTtwo^*(S_{U_1} \times_B S_{U_2}, S_{U_1} \times_B S_{U_2'} \cup S_{U_1} \vee_B S_{U_2}) \cong \HTtwo^*(S_{U_1} \wedge_B S_{U_2}, S_{U_1} \wedge_B S_{U_2'}).
    \end{align*}
    Compose with the restriction 
    \[\HTtwo^*(S_{U_1} \wedge_B S_{U_2}, S_{U_1} \wedge_B S_{U_2'}) \to \HTtwo^*(S_{U_1} \wedge_B S_{U_2}, S_{U'}),\]
    to obtain a map 
    \begin{align*}
        \smile \;: \HTtwo^*(S_{U_1}, B) &\otimes_{\HTtwo^*(B)} \HTtwo^*(S_{U_2}, S_{U_2'}) \to \HTtwo^*(S_{U}, S_{U'}) \\
        \alpha &\smile \beta := p_1^*(\alpha) \cdot p_2^*(\beta).
    \end{align*}
    Note the slight abuse of notation - the multiplication on the right is the ordinary cup product in $\HTtwo^*(S_{U}, S_{U_1} \wedge S_{U_2'})$, which is then viewed under restriction as an element of $\HTtwo^*(S_{U}, S_{U'})$.

    Let $f_1 : (S_{W_1}, S_{W_1'}) \to (S_{U_1}, S_{U_1'})$ and $f_2 : (S_{W_2}, S_{W_2'}) \to (S_{U_2}, S_{U_2'})$ be maps of pairs and write $f = f_1 \wedge f_2 : (S_W, S_{W'}) \to (S_U, S_{U'})$. Naturality of the ordinary cup product implies that the following diagram commutes.
    \begin{equation}\label{E:ExternalCupNaturality}
        \begin{tikzcd}
            \HTtwo^*(S_{W_1}, B) \otimes_{\HTtwo^*(B)} \HTtwo^*(S_{W_2}, S_{W_2'}) \arrow{r}{\smile} & \HTtwo^*(S_W, S_{W'})\\
            \HTtwo^*(S_{U_1}, B) \otimes_{\HTtwo^*(B)} \HTtwo^*(S_{U_2}, S_{U_2'}) \arrow{r}{\smile} \arrow{u}{f_1^* \otimes f_2^*} & \HTtwo^*(S_U, S_{U'}) \arrow{u}[swap]{f^*}
        \end{tikzcd}
    \end{equation}
    Hence we obtain the formula 
    \begin{align*}
        (f_1 \wedge f_2)^*(\alpha \smile \beta) &= f_1^*(\alpha) \smile f_2^*(\beta).
    \end{align*}
    Finally, let $p^{\phi}_U : (S_U, S_{U'})  \to (S_U, B)$ denote the degree one map which collapses the complement of a neighbourhood of $\phi(B)$ to infinity, as described in Definition \ref{D:RefinedThomClass}. Up to homotopy, $p^\phi_U = \id_{U_1} \wedge p^{\phi_2}_{U_2}$ hence 
    \begin{align*}
        \ttau_{U}^\phi &= (p^\phi_{U})^*\tau_U\\
        &= (\id_{U_1} \wedge p^{\phi_2}_{U_2})^*(\tau_{U_1} \smile \tau_{U_2})\\
        &= \tau_{U_1} \smile (p^{\phi_2}_{U_2})^*\tau_{U_2}\\
        &= \tau_{U_1} \smile \ttau_{U_2}^{\phi_2}.
    \end{align*}
    This proves (\ref{E:ThomSplit}).
\end{proof}
As in Lemma \ref{L:HSuvCohom}, let $\delta : \HT^*(S_{U'}, B) \to \HT^*(S_{V', U'}, S_{U'})$ be the connecting morphism in the long exact sequence of the triple $(S_{V',U'}, S_{U'}, B)$ and define $\delta_2$ similarly for the triple $(S_{V_2',U_2'}, S_{U_2'}, B)$. By the definition of the external cup product above, we have 
\begin{align*}
    \delta \tau_{U'} &= \delta(p_1^* \tau_{U_1'} \cdot p_2^* \tau_{U_2'})\\
    &= \delta p_1^* \tau_{U_1'} \cdot p_2^* \tau_{U_2'} + (-1)^{a_1'} p_1^* \tau_{U_1'} \cdot \delta p_2^* \tau_{U_2'}.
\end{align*}
Further, $\delta p_1^* \tau_{U_1'} = 0$. This can be seen from the following commutative diagram, where the bottom row is the long exact sequence of the triple $(S_{U_1'}, S_{U_1'}, B)$.
\[
\begin{tikzcd}
    ... \arrow{r} & \HT^{*-1}(S_{U'}, B) \arrow{r}{\delta} & \HT^*(S_{V', U'}, S_{U'}) \arrow{r} & \HT^*(S_{V', U'}, B) \arrow{r} & ... \\
    ... \arrow{r} & \HT^{*-1}(S_{U_1'}, B) \arrow{r}{\delta_1} \arrow{u}{p_1^*} & \HT^*(S_{U_1'}, S_{U_1'}) \arrow{r} \arrow{u}{p_1^*} & \HT^*(S_{U_1'}, B) \arrow{r} \arrow{u}{p_1^*} & ... \\
\end{tikzcd}
\]
It follows that 
\begin{align}\label{E:TauUsplitting}
    \delta \tau_{U'} &= (-1)^{a_1'}p_1^* \tau_{U_1'} \cdot  p_2^* \delta_2 \tau_{U_2'} \nonumber\\
    &= (-1)^{a_1'}\tau_{U_1'} \smile \delta_2 \tau_{U_2'}.
\end{align}

Recall that $\phi : B \to (U - U')$ is a chamber satisfying $\phi = (0, \phi_2)$ where $\phi_2 : B \to (U_2 - U_2')$ is a chamber for $f_2$. Let $f_2^*(\ttau^{\phi_2}_{V_2,U_2}) = \eta^{\phi_2}\gamma(\delta \tau_{U_2'})$ for some uniquely determined $\eta^{\phi_2} \in \HTtwo^*(\P(V_2'))$.
\begin{lemma}\label{L:i*etaCalc}
    For $\phi : B \to (U - U')$ a chamber satisfying $\phi = (0, \phi_2)$, we have 
    \begin{align*}
        (\iota_1')^*(\eta^\phi) &= 0 \\
        (\iota_2')^*(\eta^\phi) &= e(N_2)e(H_1^+)\eta_2^{\phi_2}
    \end{align*}
    where $N_2 \to \P(V_2')$ is the $\T_2$-equivariant vector bundle $N_2 = V_1 \otimes \Ocal_{\P(V_2')}(1)$.
\end{lemma}
\begin{proof}
    We start with the first equality. The inclusion map $\iota_1 : S_{V_1} \to S_V$ extends to an inclusion $\iota_1 : S_{V_1, U} = S_{V_1} \wedge_B S_U \to S_V \wedge_B S_U = S_{V, U}$ by taking the smash product of $\iota_1$ and the identity on $S_U$. Let $\iota_1'$ denote the same extension $\iota_1' : S_{V_1', U'} \to S_{V', U'}$. Let $i_{U_j'} : U_j' \to U_j$ denote inclusions for $j \in \{1,2\}$. There is a map $f_1 \wedge i_{U_2'} : S_{V_1', U'} \to S_{V_1, U}$ which is defined below.
    \begin{equation}\label{E:f1wedgei2}
        \begin{tikzcd}
         S_{V_1', U'} = S_{V_1', U_1'} \wedge_B S_{U_2'} \arrow{r}{f_1 \wedge i_{U_2'}} & S_{V_1, U_1} \wedge_B S_{U_2} = S_{V_1, U} 
        \end{tikzcd}
    \end{equation}
    Since $f_2|_{U_2'} = i_{U_2'}$, the following diagram commutes.
    \[
    \begin{tikzcd}
        S_{V', U'} \arrow{r}{f_1 \wedge f_2} & S_{V, U}\\
        S_{V_1', U'} \arrow{u}{\iota_1'} \arrow{r}{f_1 \wedge i_{U_2'}} & S_{V_1, U} \arrow[swap]{u}{\iota_1}
    \end{tikzcd}
    \]
    Note that each map fixes $S_{U'}$, hence we obtain the below diagram in equivariant cohomology.
    \begin{equation}\label{E:H*diagram}
        \begin{tikzcd}
            \HT^*(S_{V', U'}, S_{U'}) \arrow[swap]{d}{(\iota_1')^*} &  \HT^*(S_{V, U}, S_{U'}) \arrow[swap]{l}{(f_1 \wedge f_2)^*} \arrow{d}{\iota_1^*}\\
            \HT^*(S_{V_1', U'}, S_{U'})   & \HT^*(S_{V_1, U}, S_{U'}) \arrow{l}{(f_1 \wedge i_{U_2'})^*} 
        \end{tikzcd}
    \end{equation}
    Let $\ttau^\phi_{V_1, U} \in \HT^*(S_{V_1, U}, S_{U'})$ be the refinement of $\tau_{V_1, U} \in \HT^*(S_{V_1, U}, B)$ induced by $\phi$. By Lemma \ref{L:RefinedThomSplit} we have 
    \begin{align*}
        \iota_1^*(\ttau_{V,U}^\phi) &= \iota_1^*(\tau_{V_2} \smile \ttau^\phi_{V_1, U})\\
        &= e(V_2) \cdot \ttau^\phi_{V_1, U}
    \end{align*}
    Note that $\tau_{V_2} \in \HTtwo^{2a_2}(S_{V_2}, B)$ has even degree, hence the sign is correct. The map $(f_1 \wedge i_{U_2'})^*$ factors as follows 
    \begin{equation}
        \begin{tikzcd}
            \HT^*(S_{V_1', U'}, S_{U'}) & \arrow{l}[swap]{(f_1 \wedge \id)^*} \HT^*(S_{V_1, U_1} \wedge_B S_{U_2'}, S_{U'}) & \arrow{l}[swap]{(\id \wedge i_{U_2'})^*} \HT^*(S_{V_1, U}, S_{U'}) \arrow[bend right=25,swap]{ll}{(f_1 \wedge i_{U_2'})^*}
        \end{tikzcd}
    \end{equation}
    Again $\ttau^\phi_{V_1, U} = \tau_{V_1, U_1} \smile \ttau^{\phi_2}_{U_2}$ where $\ttau^{\phi_2}_{U_2} \in \HT^*(S_{U_2}, S_{U_2'})$ is the refined Thom class induced by $\phi_2$. As in (\ref{E:ExternalCupNaturality}), the naturality of the external cup product gives the following commutative diagram.
    \[
    \begin{tikzcd}
        \HT^*(S_{V_1, U_1} \wedge_B S_{U_2'}, S_{U'}) & \arrow{l}[swap]{(\id \wedge i_{U_2'})^*} \HT^*(S_{V_1, U}, S_{U'})\\
        \HT^*(S_{V_1, U_1}, B) \otimes_{\HT^*(B)} \HT^*(S_{U_2'}, S_{U_2'}) \arrow{u}{\smile} & \arrow{l}[swap]{\id \otimes i_{U_2'}^*} \HT^*(S_{V_1, U_1}, B) \otimes_{\HT^*(B)} \HT^*(S_{U}, S_{U'}) \arrow{u}[swap]{\smile}
    \end{tikzcd}
    \]
    Since $\HT^*(S_{U_2'}, S_{U_2'}) = 0$, it follows that 
    \begin{align*}
        (\id \wedge i_{U_2'})^*(\iota_1^* \ttau^\phi_{V,U}) &= (\id \wedge i_{U_2'})^*(e(V_2)\ttau_{V_1, U}^\phi) \\
        &= e(V_2)(\id\wedge i_{U_2'})^*(\tau_{V_1, U_1} \smile \ttau_{U_2}^{\phi_2})\\
        &= 0.
    \end{align*}
    It follows from the commutativity of (\ref{E:H*diagram}) that 
    \begin{align*}
        ((\iota_1')^*\eta^\phi)\delta \tau_{U'} &= (\iota_1')^*(f^* \ttau_{V,U}^\phi)\\
        &= (f_1 \wedge i_{U_2'})^*(\iota_1^* \tau_{V,U}^\phi)\\
        &= 0.
    \end{align*}
    Consequently $(\iota_1')^*\eta^\phi = 0$.

    For the second equality, there is a similar commutative diagram.
    \[
    \begin{tikzcd}
        \HT^*(S_{V', U'}, S_{U'}) \arrow[swap]{d}{(\iota_2')^*} &  \HT^*(S_{V, U}, S_{U'}) \arrow[swap]{l}{(f_1 \wedge f_2)^*} \arrow{d}{\iota_2^*}\\
        \HT^*(S_{V_2', U'}, S_{U'})   & \HT^*(S_{V_2, U}, S_{U'}) \arrow{l}{(i_{U_1'} \wedge f_2)^*}
    \end{tikzcd}
    \]
    We have that $\iota_2^*\ttau^\phi_{V,U} = e(V_1)\ttau^\phi_{V_2, U}$ for $\ttau^\phi_{V_2, U} \in \HT^*(S_{V_2, U}, S_{U'})$. Lemma \ref{L:RefinedThomSplit} ensures that $\ttau^\phi_{V_2, U}$ factors as $\ttau^\phi_{V_2, U} = \tau_{U_1} \smile \ttau^{\phi_2}_{V_2, U_2}$ under the external cup product 
    \begin{align*}
        \smile \; : \HT^*(S_{U_1}, B) \otimes_{\HT^*(B)} \HT^*(S_{V_2, U}, S_{U'}) \to \HT^*(S_{V_2, U}, S_{U'}).
    \end{align*}
    The calculation continues as follows.
    \begin{align*}
        ((\iota_2')^*\eta^\phi)\delta \tau_{U'} &= (\iota_2')^*(f^* \ttau_{V,U}^\phi)\\
        &= (i_{U_1'} \wedge f_2)^*(\iota_2^* \ttau_{V,U}^\phi)\\
        &= (i_{U_1'} \wedge f_2)^*(e(V_1)\ttau^\phi_{V_2, U})\\
        &= e(V_1)(i_{U_1'} \wedge f_2)^*(\tau_{U_1} \smile \ttau^{\phi_2}_{V_2, U_2})\\
        &= e(V_1)i_{U_1'}^*(\tau_{U_1}) \smile f_2^*(\ttau^{\phi_2}_{V_2, U_2})
    \end{align*}
    Recall that $H_1^+$ is the fibre of the bundle $U_1 \to U_1'$. Applying the formula (\ref{E:TauUsplitting}) we obtain
    \begin{align*}
        ((\iota_2')^*\eta^\phi)\delta \tau_{U'} &= e(V_1)e(H_1^+)\tau_{U_1'} \smile \eta_2^{\phi_2} \delta \tau_{U_2'}\\
        &= e(V_1)e(H_1^+)\eta_2^{\phi_2} \delta \tau_{U'}
    \end{align*}
    Thus $(\iota_2')^*\eta^\phi = e(V_1)e(H_1^+)\eta_2^{\phi_2}$. Interpreting $V_1$ as a vector bundle over $\P(V_2')$, the product $\P(V_2') \times_B V_1 \to \P(V_2')$ can be identified as the normal bundle $N_2 \to \P(V_2')$ of $\P(V_2') \subset \P(V_2' \oplus V_1)$. Applying the same argument as Lemma \ref{L:NiCalc} shows that $N_2$ is equivariantly isomorphic to $(V_1 \otimes \Ocal_{\P(V_2')}(1))_1$. Since $e(V_1)$ is the $\T$-equivariant Euler class of $V_1$, it follows that $(\iota_2')^*\eta^\phi = e(N_2)e(H_1^+)\eta_2^{\phi_2}$
\end{proof}
Let $M_2$ denote the virtual vector bundle $M_2 = N_2' - N_2$ over $\P(V_2)$. Notice that 
\begin{align}\label{E:M2Def}
M_2 &= V_1' \otimes \Ocal_{\P(V_2')}(1) - V_1 \otimes \Ocal_{\P(V_2')}(1) \nonumber\\
 &= D_1 \otimes \Ocal_{\P(V_2')}(1).
\end{align}
Here $D_1 = V'_1 - V_1$ is a virtual vector bundle over $\P(V_2)$ of rank $d_1 = a_1' - a_1$. By Lemma \ref{L:e(N_i)calc} we have
\begin{align*}\label{E:M2Euler}
    e(M_2)\inv = \sum_{j \geq 0} s_j(D_1 \otimes \Ocal_{\P(V_2')}(1))y^{-d_1 - j}.
\end{align*}
Note that this sum is finite since the $j$-th Segre class is an element of $H^*(B)$ and $B$ is a finite dimensional manifold.
\begin{lemma}\label{L:SegreCalcD1}
    The $j$-th Segre class of the virtual bundle $D_1 \otimes \Ocal_{\P(V_2)}(1)$ is given by the formula
    \begin{align*}
        s_j(D_1 \otimes \Ocal_{\P(V_2)}(1)) &= \sum_{l=0}^{j} {{-d_1-l}\choose {j - l}}s_l(D_1)x^{j-1}
    \end{align*}
    where $x = c_1(\Ocal_{\P(V_2)})$.
\end{lemma}
\begin{proof}
    Suppose that $V \to B$ is an $n$-dimensional vector bundle with Chern roots $\alpha_1, ..., \alpha_n$. Then by the splitting principle, the Chern roots of $V \otimes \Ocal_{\P(V_2)}(1)$ are $\alpha_1 + x, ..., \alpha_n + x$ and we have 
    \begin{align*}
        c_j(V \otimes \Ocal_{\P(V_2)}(1)) &= \sum_{l=0}^j c_l(V) x^{j - l} {{n - l} \choose {j - l}}.
    \end{align*}
    This formula also holds if $V = V_1 - V_2$ is a rank $n$ virtual bundle. In this case, $s(V) = c(V)\inv = c(-V)$ and $s_j(V) = c_j(-V)$. Applying this formula to the virtual vector bundle $D_1$ of rank $d_1$, we obtain
    \begin{align*}
        s_j(D_1 \otimes \Ocal_{\P(V_2)}(1)) &= c_j((-D_1) \otimes \Ocal_{\P(V_2)}(1))\\
        &= \sum_{l=0}^{j} {{-d_1-l}\choose {j - l}}c_l(-D_1)x^{j-1}\\
        &= \sum_{l=0}^{j} {{-d_1-l}\choose {j - l}}s_l(D_1)x^{j-1}.
    \end{align*} 
\end{proof}
The map $f_1 : S_{V_1',U_1'} \to S_{V_1, U_1}$ induces a map $f_1^* : \HTtwo(S_{V_1, U_1}, B) \to \HTtwo(S_{V_1', U_1'}, B)$ which sends $\tau_{V_1, U_1}$ to some $\HTtwo^*(B)$-multiple of $\tau_{V_1', U_1'}$ in $\HTtwo(S_{V_1', U_1'}, B)$. Define the degree of $f_1$ to be the class $\degSone(f_1) \in \HTtwo(S_{V_1', U_1'}, B)$ such that 
\begin{align}
    f_1^* (\tau_{V_1, U_1}) &= \degSone(f_1) \tau_{V_1', U_1'}.
\end{align}
To calculate $\degSone(f_1)$, consider the following commutative diagram.
\[
\begin{tikzcd}
    S_{V'_1, U'_1} \arrow{r}{f_1} & S_{V_1,U_1}\\
    S_{U'_1} \arrow{u}{i_{U'_1}} \arrow{r}{f_1|_{U'_1}} & S_{U_1} \arrow{u}[swap]{i_{U_1}}
\end{tikzcd}
\]
Recall that $f|_{U'}$ is the inclusion and that $U_1 = U_1' \oplus H_1^+$, hence 
\begin{align*}
    \degSone(f_1) e(V_1') \tau_{U_1'} &= i_{U'}^*(\degSone(f_1) \tau_{V_1', U_1'})\\
    &= i_{U'}^*(f_1^* \tau_{V_1, U_1})\\
    &= f_1|_{U_1'}^* (i_{U_1}^*\tau_{V_1, U_1})\\
    &= f_1|_{U_1'}^* (e(V_1)\tau_{U_1})\\
    &= e(H_1^+)e(V_1)\tau_{U_1'}.
\end{align*}
It follows that 
\begin{align*}
    \degSone(f_1) &= e(H_1^+)e(-D_1)
\end{align*}
where $D_1$ is the virtual bundle $D_1 = V_1' - V_1$. The calculation in the proof of Lemma \ref{L:NiCalc} extends to virtual bundles and shows that
\begin{align}\label{E:DegSonefone}
    \degSone(f_1) &= e(H^+_1)\sum_{l=0}^{-d_1} s_l(D_1)x^{-d_1 - l}.
\end{align}
This formula for the degree of $f_1$ can be applied to derive cohomological proof of Donaldson's theorem \cite[Theorem 3.1]{BaragliaConstraints}. We will use it to calculate the Seiberg-Witten invariant of a smash product, from which the Seiberg-Witten connected sum formula follows immediately. Theorem \ref{T:FSWofWedge} and Theorem \ref{T:FSW-CSF} below are new results.
\begin{theorem}\label{T:FSWofWedge}
    Suppose $f_i : S_{V_i', U_i'} \to S_{V_i, U_i}$ are finite dimensional monopole maps for $i \in \{1,2\}$ and that $\phi_2 : B \to (U_2 - U_2')$ is a chamber for $f_2$. Let $\phi = (0, \phi_2)$ be a chamber for the monopole map $f = f_1 \wedge_B f_2 : S_{V', U'} \to S_{V, U}$ and write $U_1 = U_1' \oplus H_1^+$. Then 
    \begin{align}
        \SWfphi_m &= \SWfphitwo(x^m \degSone(f_1)).
    \end{align}
\end{theorem}
\begin{proof}
    From (\ref{E:SWLocCalc}) we have 
    \begin{align*}
        \SWhatfphi_m &= (\pi_1')_*(x^m e(N_1')\inv (\iota_1')^*\teta^\phi) + (\pi_2')_*(x^m e(N_2')\inv (\iota_2')^*\teta^\phi).
    \end{align*}
    Applying Lemma \ref{L:i*etaCalc} gives
    \begin{align*}
        \SWhatfphi_m &= (\pi_2')_*(x^m e(N_2')\inv e(N_2)e(H_1^+)\eta^{\phi_2})\\
        &= \SWhat^{f_2, \phi_2}(x^m e(H_1^+)e(M_2)\inv )
    \end{align*}
    Note that $M_2 = N_2 - N_2'$ is a complex virtual bundle and therefore has even rank. Apply (\ref{E:M2Euler}) to obtain
    \begin{align*}
        \SWhatfphi_m &= \SWhat^{f_2, \phi_2}\left(x^m e(H_1^+)\sum_{j \geq 0} s_j(D_1 \otimes \Ocal_{\P(V_2')}(1))y^{-d_1 - j} \right)\\
        &= \sum_{j \geq 0} \SWhat^{f_2, \phi_2}\left(x^m e(H_1^+) s_j(D_1 \otimes \Ocal_{\P(V_2')}(1)) \right)y^{-d_1 - j}.
    \end{align*}
    Recall that the ordinary Seiberg Witten invariant $\SWfphi_m$ is given by evaluating the $\T_2$-generalised Seiberg-Witten invariant at $y = 0$. It follows that 
    \begin{align*}
        \SWfphi_m &= SW^{f_2, \phi_2}\left(x^m e(H_1^+)s_{-d_1}(D_1 \otimes \Ocal_{\P(V_2')}(1)) \right).
    \end{align*}
    Apply Lemma \ref{L:SegreCalcD1} and equation (\ref{E:DegSonefone}) to obtain 
    \begin{align*}
        \SWfphi_m &= SW^{f_2, \phi_2}\left(x^m e(H_1^+)\left(\sum_{l=0}^{-d_1} s_l(D_1)x^{-d_1-1}\right)\right) \\
        &= \SWfphitwo(x^m \degSone(f_1)).
    \end{align*}
\end{proof}
Substituting in $\degSone(f_1)$ gives an alternative presentation of the formula
\begin{align*}
    \SWfphi_m &= e(H_1^+)\sum_{l=0}^{-d_1} s_l(D_1) SW^{f_2, \phi_2}(x^{m-d_1-1}).
\end{align*}
\begin{theorem}[Families Seiberg-Witten Connected Sum Formula]\label{T:FSW-CSF}
    For $j \in \{1,2\}$, let $E_j \to B$ be a 4-manifold family equipped with a \spinc structure $\sfrak_j$ on the vertical tangent bundle. Let $i_j : B \to E_j$ be a section with normal bundle $V_j$ and assume that $\vphi : V_1 \to V_2$ is an orientation reversing isomorphism satisfying $\vphi(i_1^*(\sfrak_{E_1})) \cong i_2^*(\sfrak_{E_2})$. Set $E = E_1 \#_B E_2$ and let $\phi_2$ be a chamber for $\mu_2$ with $\phi = (0, \phi_2)$ defining a chamber for $\mu$. Then 
    \begin{align}\label{E:FSW-CSF}
        SW^{\mu, \phi}_m &= SW^{\mu_2,\phi_2}(x^m \degSone(\mu_1)).
    \end{align}
\end{theorem}
\begin{proof}
    Theorem \ref{T:FamBFCSF} gives $[\mu] = [\mu_1] \wedge_\Jcal [\mu_2]$ and thus the result follows immediately from Theorem \ref{T:FSWofWedge}.
\end{proof}
Theorem \ref{T:FSW-CSF} has many potential applications, one such instance being the ability to detect diffeomorphisms which are topologically isotopic to the identity, but not smoothly isotopic. Let $f : X \to X$ be a diffeomorphism and let $E_f \to S^1$ denote the mapping torus of $f$, which is a 4-manifold family over the circle. Assume that $X$ is simply connected. By work of Freedman and Quinn, $f$ is topologically isotopic to the identity if and only if it acts trivially on $H^2(X ; \Z)$ \cite{Quinn}. This is an easily verifiable condition, however detecting smooth isotopy is more difficult. 

One obstruction to $f$ being smoothly isotopic to the identity is the smoothability of $E_f$. This can be detected by the families Seiberg-Witten invariant. This approach has been taken by both Ruberman \cite{SmoothIsotopyObstruction} and Baraglia-Konno \cite{BaragliaSWGluingFormula}. Since $b_1(X) = 0$, it is necessary that $b^+(X)$ be even and greater than $2$ to have a non-zero families Seiberg-Witten invariant of $E_f$. Baraglia-Konno constructed examples using their formula with $b^+(X) \geq 4$, which was necessary to avoid chambers. Theorem \ref{T:FSW-CSF} can accommodate chambers, hence has potential to construct examples with $b^+(X) = 2$.

\bibliography{Bibliography}{}
\bibliographystyle{plain}

\end{document}